%% file: ms.tex
\newif\ifsinum
\newif\ifesaim
\newif\ifhal

\haltrue   

\ifhal
  \documentclass[a4paper, 10pt]{article}
\else
    \ifesaim
      \documentclass{m2an}
    \else
      \documentclass[onefignum,onetabnum]{siamart171218}
    \fi
\fi

\usepackage{blindtext}
\usepackage{lscape}
\usepackage{gensymb}
\usepackage{pdflscape}
\usepackage{geometry}
\usepackage{tabu,array}
\newcolumntype{?}{!{\vrule width 1pt}}
\usepackage{multirow}
\usepackage{lipsum}
\usepackage{pdfpages}
\usepackage{amsfonts}
\usepackage{graphicx}
\usepackage{epstopdf}
\usepackage{algorithmic}
\usepackage{enumerate}
\usepackage{diagbox}
\usepackage{bm}
\usepackage{float,  stmaryrd}
\usepackage{mathtools}
\usepackage{amsmath, amssymb}
\usepackage{geometry, marginnote}
\usepackage{color,  subcaption}
\usepackage[normalem]{ulem} 
\usepackage{cancel} 
\usepackage[backgroundcolor=lightgray]{todonotes}
\usepackage{amsopn}
\usepackage{amsthm}

\usepackage{hyperref}
\hypersetup{
    colorlinks,          
    linkcolor=blue,
    citecolor=blue,        
    filecolor=magenta,      
    urlcolor=black,           
}

\usepackage{longtable, makecell}

\input{command}
\marginfalse
\detailfalse

\begin{document}

\ifhal
  \input{ex_shared_hal}
\else
    \ifesaim
      \input{ex_shared_m2an}
    \else
      \input{ex_shared_sinum}
    \fi
\fi


%


\section{Introduction}
This work deals with a linear advection equation of the form: find $\uu: \dom \subset \R^\dimx \rightarrow \R$ such that
\begin{subequations}
  \label{Eq::adv_multiD}
  \begin{align}
    \inprod{\vel}{\grad \uu} = \ff, & \qquad \text{ in } \dom,   \\
    \uu = 0,                        & \qquad \text{ on } \bdomm.
  \end{align}
\end{subequations}
The velocity field $\vel \in \Cnt^1(\overline \dom; \R^\dimx)$, $\vel \neq 0$,  is considered to be divergence-free and we take into account a general source term $\ff \in \Ltwo(\dom)$.
The inflow, outflow, and characteristic parts of the boundary are denoted by $\bdomm,\bdomp$, and $\bdomz$, respectively, with the definitions
$$ \partial_\pm \dom \defeq \{x \in \bdom: \pm \vel(x) \scp \nn(x) > 0  \}, \qquad \bdomz \defeq \{x \in \bdom: \vel(x) \scp \nn(x) = 0  \}.$$
In the main body of the paper, we focus on the one-dimensional case $\dimx = 1$, where $\dom \subset \R$ is a bounded simply connected interval; then $\vel$ is a constant scalar. We keep the notations in multi-dimensional form in order to be applicable when we discuss extensions of our results to the multi-dimensional case.

The a posteriori error analysis for  problem~\eqref{Eq::adv_multiD} admits a range of functional frameworks and consequently different norms in which the error can be measured. Our goal is to derive an $\Ltwo$-norm error estimate of the form
\begin{equation}
  \label{Eq::general_up}
  \Vert \uu - \uh \Vert_{\Ltwo(\dom)} \leq \Err,
\end{equation}
where $\uu$ is the weak solution of~\eqref{Eq::adv_multiD} in $\Ltwo(\dom)$, $\uh$ is its numerical approximation, and $\Err$ is an a posteriori error estimator {\em fully computable} from $\uh$ by some {\em local procedure}. We seek to have a bound that is {\em guaranteed}, \ie, featuring no unknown constant, in contrast to {\em reliability} where a bound up to a generic constant is sufficient.
We develop a unified framework treating several classical numerical methods at once. Importantly, we also prove a converse estimate to~\eqref{Eq::general_up} in the form
\begin{equation}
  \label{Eq::general_low}
  \Err \leq C \Vert \uu - \uh \Vert_{\Ltwo(\dom)} + \text{data oscillation}.
\end{equation}
This is called \emph{global efficiency} and yields equivalence between the incomputable error $\Vert \uu - \uh \Vert_{\Ltwo(\dom)}$ and the computable estimator $\Err$ up to the data oscillation term that vanishes for piecewise polynomial datum $\ff$ and that is of higher order than the error for piecewise smooth datum $\ff$. Crucially, in our developments, the generic constant $C$ in~\eqref{Eq::general_low} only depends on the mesh shape regularity, requesting for $d=1$ each two neighboring elements to be of comparable size. In particular, $C$ is independent of the problem parameters $\vel$ and $\ff$ as well as of the polynomial degree of the approximation $\poldeg$, yielding both data- and polynomial-degree-\emph{robustness}. We actually also show local efficiency, \ie, a localized version of~\eqref{Eq::general_low}, which is highly desirable on the practical side in view of adaptive mesh refinement.

To achieve the above-mentioned goals, we start with the ultra-weak variational formulation at the infinite-dimensional level, where the solution lies in the $\Ltwo(\dom)$ trial space and the test space is formed by $\Hs{1}(\dom)$ functions taking zero value at the outflow boundary. In this setting, we prove the equality of the $\Ltwo$-norm of the error with the dual graph norm (relying on $\Vert \inprod{\vel}{\grad (\cdot)}\Vert_{\Ltwo(\dom)}$) of the residual.
In the one-dimensional case, we are able to prove that the global dual norm can be localized over vertex-based patches of elements under an orthogonality condition against the hat basis functions. Consequently, suitable discrete local problems posed over these patches are identified which lead to local reconstructions $\sha$ combined into a global reconstruction $\sh$ such that $\Vert \uh - \sh\Vert_{\Ltwo(\dom)}$ forms the main ingredient of the estimator $\Err$ satisfying~\eqref{Eq::general_up} and~\eqref{Eq::general_low}.



Let us recall some important contributions to a posteriori error estimation for problem~\eqref{Eq::adv_multiD}. Bey and Oden in~\cite{bey1996hp} proposed an a posteriori error estimate for a discontinuous Galerkin ({\DG}) formulation of the multi-dimensional advection--reaction problem. In this framework, two infinite-dimensional problems have to be solved on each mesh element; one to obtain the lower bound on the error and one for the upper bound, in two different and inequivalent weighted energy norms. This gives estimates similar to~\eqref{Eq::general_up} and \eqref{Eq::general_low}, but for two different estimators and  in two different norms of the error. Additionally, one cannot solve analytically the infinite-dimensional elementwise problems, and, in practice, one needs to approximate them by some higher-order finite element approximation. Hence, neither simultaneous reliability and efficiency, nor robustness, are granted.


S\"uli in~\cite{suli1999posteriori} applied the $\Hs{1}$-stability result of Tartakoff~\cite{tartakoff1972regularity} to the adjoint problem of~\eqref{Eq::adv_multiD} (with the presence of the reaction term and in the multi-dimensional case), and obtained a global reliable upper bound on the $\Hs{-1}$-norm of the error in terms of the $\Ltwo$-norm of the residual for a weak formulation of~\eqref{Eq::adv_multiD} with distinct trial and test spaces. He further turned this bound into a reliable $\Hs{-1}$-norm a posteriori error indicator for the streamline-diffusion finite element and the cell-vertex finite volume methods.
However, neither the efficiency nor the robustness of this error indicator are discussed.
Furthermore, in~\cite{suli1999posteriori} by S\"uli  and in~\cite{houston1999posteriori} by S\"uli and Houston, an analysis of the multi-dimensional advection--reaction problem in the graph space equipped with the full norm $\Ltwonorm{\cdot}{{\Ltwo(\dom)}} + \Ltwonorm{\inprod{\vel}{\grad (\cdot)}}{{\Ltwo(\dom)}}$,  is provided. This functional setting provides the equivalence between the $\Ltwo$-norm of the error and the dual graph norm of the residual up to some generic constants, which is a weaker error--residual equivalence result compared to what we establish in the current work, see Theorem~\ref{Thm::Err_Res} below, upon replacing the full graph norm by $\Ltwonorm{\inprod{\vel}{\grad (\cdot)}}{{\Ltwo(\dom)}}$ only.
Overall, no complete reliability, efficiency, and robustness results of the form~\eqref{Eq::general_up}--\eqref{Eq::general_low} are obtained.




Becker \eal in~\cite{becker2013reconstruction}, derived reconstruction-based error estimators for the advection problem~\eqref{Eq::adv_multiD} in two space dimensions. An $\Hdiv{\dom}$-conforming reconstruction is proposed for the flux vector $\vel \uu$  (instead of $\uu$ in the present work) which is designed to produce a guaranteed upper bound on the error measured in some dual norm of the advection operator. A unified framework is built, covering the {\DG} and conforming finite element methods with/without stabilization terms. This dual norm is hard to evaluate even for a known exact solution, and, in practice, the authors replace it by the $\Ltwo$-norm, so that the guaranteed upper bound property is eventually lost. Proofs of efficiency or robustness are not given, but optimal convergence orders of the estimator are observed in numerical experiments.
It is worth mentioning that, restricted to one space dimension, the dual norm of~\cite{becker2013reconstruction} reduces to the weak graph norm we employ. Our contribution in this respect consists in the proofs of~\eqref{Eq::general_up} and~\eqref{Eq::general_low}, not given in~\cite{becker2013reconstruction} (where, recall, two space dimensions are treated.)


In a recent result by Georgoulis \eal in \cite{georgoulis2019posteriori},  the authors used the reconstruction proposed by Makridakis and Nochetto in~\cite{makridakis2006posteriori} for a {\DG} approximation and provided a reliable upper bound on the error in the energy norm for one-dimensional advection--diffusion--reaction problems, as well as a reliable $\Ltwo$-norm estimate for the  problem~\eqref{Eq::adv_multiD} in one space dimension. Though a proof of~\eqref{Eq::general_low} is not given, efficiency and robustness are numerically observed.
One might also note the earlier work of these authors~\cite{georgoulis2014error}, dedicated to the two-dimensional advection--reaction problem with a similar reconstruction. In that work, a reliable bound on the energy norm of the error is presented, though again without a theoretical elaboration on the efficiency and robustness.


Finally, let us mention the recent result of Dahmen and Stevenson in~\cite{dahmen2019adaptive} where the authors provide a posteriori error estimates for the discontinuous Petrov--Galerkin method tailored to the transport equations in two space dimensions.
The equivalence of the errors of the bulk and skeleton quantities with the dual norm of the residual is established. This dual norm is later  approximated by some equivalent yet computable indicator. The absorbed constants translate into a constant $C$ in~\eqref{Eq::general_low} which depends on the advective field $\vel$ and the polynomial degree of approximation, therefore precluding the robustness of the error lower bound.

We also mention that in the case of advection--diffusion(--reaction) problems, other approaches were previously considered to obtain robustness with respect to the advective field. Among them, Verf{\"u}rth~\cite{verfurth2005robust} proposed to augment the energy norm by a dual norm coming from the skew-symmetric part of the differential operator, and Sangalli~\cite{sangalli2004analysis,sangalli2005uniform} used interpolated spaces and a fractional-order norm for the advective term. Extensions of these approaches can be found in~\cite{sangalli2008robust, schotzau2009robust, ern2010guaranteed}. However, the above results are not applicable when the diffusion parameter vanishes, \ie, as the advection--diffusion problem reduces to~\eqref{Eq::adv_multiD} because the diffusive part of the operator is needed to evaluate the dual norm.

We treat problem~\eqref{Eq::adv_multiD} in one space dimension in Sections~\ref{Sec::inf_sup}--\ref{sec:experiments}.
Section~\ref{Sec::inf_sup} deals with the functional settings, whereby adopting the ultra-weak variational formulation. We prove, in particular, the equality of the $\Ltwo$-norm of the error and the dual norm of the residual. Section~\ref{Sec::Example} introduces some numerical schemes for approximating~\eqref{Eq::adv_multiD}. Section~\ref{Sec::localization} discusses the localization of the dual norm of the residual over vertex-based patches, showing in particular that this is possible for the schemes discussed in Section~\ref{Sec::Example}. In Section~\ref{Sec::Local_recons}, we present our local patchwise reconstruction. Sections~\ref{Sec::upper_bound} and~\ref{Sec::efficiency} then present the proofs for the upper and lower bounds as well as robustness in the form of~\eqref{Eq::general_up}--\eqref{Eq::general_low}. Section~\ref{sec:experiments} then contains results of several numerical experiments to illustrate the developed theory. Finally, in Section~\ref{Sec::extend}, we {consider the advection problem~\eqref{Eq::adv_multiD} in multiple space dimensions and derive a heuristic extension of our methodology to this case. Although we cannot prove here the guaranteed upper bound, (local) efficiency, and robustness, numerical experiments indicate appreciable properties of the derived estimates also in this case.

\section{Abstract framework}
\label{Sec::inf_sup}
We start with the presentation of the abstract framework.
\subsection{Spaces}
In the one-dimensional case, the constraint of $\vel$ being a non-zero divergence-free field is translated to $\vel$ being a constant nonzero scalar. Consequently, we are lead to work with the spaces
\begin{equation} \label{eq_Hpm}
  \Hbminus{\dom}    = \lf\{ \ww \in \Hs{1}(\dom),  \ww    = 0, \text{ on } {\bdomm} \rt\},  \qquad
  \Hbplus{\dom}  = \lf\{ \ww \in \Hs{1}(\dom), \ww  = 0, \text{ on } {\bdomp}  \rt\}.
\end{equation}
The trace operator in these spaces is well-defined and the following \emph{integration-by-parts} formula holds:
\begin{equation}
  \label{Eq::IBP}
  \integprod{\vv}{\inprod{\vel}{\grad \ww}}{\dom} + \integprod{\inprod{\vel}{\grad \vv}}{\ww}{\dom} = \integprod{\inprod{\vel}{\nn} \vv}{\ww}{\bdom}, \qquad \forall \vv, \ww \in \Hs{1}(\dom),
\end{equation}
where the notation $\integprod{\vv}{\ww}{\DD} \defeq \int_\DD \vv \ww$ is used for an open subdomain $\DD \subseteq \dom$ or its boundary $\partial \DD$ and for integrable functions $\vv$ and $\ww$.
Henceforth, $\Ltwonorm{\vv}{\DD}$ denotes the norm $\Vert \vv \Vert_{\Ltwo(\DD)} = \sqrt{\integprod{\vv}{\vv}{\DD}}$. We will drop the subscript when $\DD = \dom$.

\subsection{Poincar\'e inequalities}
The Poincar\'e inequality states that
\begin{subequations}\label{Eq::PF}
\begin{equation}
  \label{Eq::Poincare}
  \Ltwonorm{ \vv - \bar \vv}{\DD} \leq \h_\DD  \cp{\DD} \Ltwonorm{ \grad \vv }{\DD}, \qquad \forall \vv \in \Hs{1}(\DD),
\end{equation}
with $\cp{\DD} > 0$ a generic constant, in particular equal to $1/\pi$ for convex $\DD \subset \Omega$. Here $\bar \vv$ is the mean value of $\vv$ over $\DD$ defined as $\bar \vv = \integprod{\vv}{1}{\DD}/\vert \DD \vert$ and  $\h_\DD$ is the diameter of $\DD$. Similarly, another Poincar\'e inequality (sometimes called Friedrichs inequality) states that
\begin{equation}
  \label{Eq::Friedrichs}
  \Ltwonorm{ \vv }{\DD} \leq \h_\DD  \cF{\DD, \partial_0 \DD} \Ltwonorm{ \grad \vv }{\DD}, \qquad \forall \vv \in \lf \{ \Hs{1}(\DD), \vv \vert_{\partial_0 \DD} = 0, \vert \partial_0 \DD \vert \neq 0  \rt \},
\end{equation}\end{subequations}
where $\partial_0 \DD \subset \partial \DD$; typically $\cF{\DD, \partial_0 \DD} = 1$.
Henceforth, we will use $\cpf{\DD}$ as a general notation for both $\cp{D}$ and $\cF{D, \partial_0 D}$. It follows from the above that for a one-dimensional interval $D$, $\cpf{D}$ can be taken as $1$.

\subsection{Variational formulation and residual}
The variational framework hinges upon an appropriate choice of the trial and test spaces and their corresponding norms.  In particular, it turns out natural to work on spaces well-suited to the non-symmetric structure of the problem. Here we consider Hilbert spaces (non-symmetric formulations in Banach spaces can be found in~\cite{cantin2017well,muga2018discrete}).

The (usual) weak formulation of~\eqref{Eq::adv_multiD} reads: find $\uu \in \Hbminus{\dom}$ such that
  \begin{equation}
    \label{Eq::weak_one}
    \integprod{\inprod{\vel}{\grad \uu}}{\vv}{} = \integprod{\ff}{\vv}{}, \qquad \forall \vv \in \Ltwo(\dom).
  \end{equation}
It is classically well-posed as one might confer with~\cite{friedrichs1958symmetric}, \cite{lax1960local}, and \cite[Proposition 6]{suli1999posteriori}, \cf also \cite{ern2006discontinuous} and \cite[Rem. 2.2]{dahmen2012adaptive}.  Here, we rather adopt the so-called \emph{ultra-weak} formulation of problem~\eqref{Eq::adv_multiD} where the bilinear form is obtained by casting the derivatives on the test function, using integration-by-parts. It reads: find $\uu \in \Ltwo(\dom)$ such that
\begin{equation}
  \label{Eq::weak_two}
  - \integprod{\uu}{\inprod{\vel}{\grad \vv}}{}  =  \integprod{\ff}{ \vv}{}, \qquad \forall \vv \in \Hbplus{\dom}.
\end{equation}
The well-posedness of~\eqref{Eq::weak_two} can be shown by inf--sup arguments (\cf~\cite[Theorem 2.6]{ern2006discontinuous} and~\cite[Theorem 2.4]{dahmen2012adaptive}).


Denote by $\Hbplusdual{\dom}$ the dual space to $\Hbplus{\dom}$. For an arbitrary $\uh \in \Ltwo(\dom)$, the formulation~\eqref{Eq::weak_two} leads to the definition of the residual $\Res(\uh)$, a bounded linear functional on $\Hbplusdual{\dom}$, by
\begin{equation}
  \label{Eq::Res_two}
  \pair{\Res(\uh)}{\vv} \defeq (\ff, \vv) + (\uh, \inprod{\vel}{\grad \vv}), \qquad \forall \vv \in \Hbplus{\dom}.
\end{equation}
We define its velocity-scaled dual norm by
\begin{equation}
  \label{Eq::Res_two_norm}
  \Vert \Res(\uh) \Vert_{\vel; \, \Hbplusdual{\dom}} \defeq \sup_{{\vv \in \Hbplus{\dom}\setminus \{0\}}} \frac{\pair{\Res(\uh)}{\vv}}{{\Ltwonorm{\inprod{\vel}{\grad \vv}}{}}}.
\end{equation}


\subsection{Error-residual equivalence}
In this section, we present an important connection between the $\Ltwo(\dom)$-norm of the error and the residual {norm~\eqref{Eq::Res_two_norm}}:
\begin{thm}[error-residual equivalence]
  \label{Thm::Err_Res} Let $\uu$ be the ultra-weak solution  of~\eqref{Eq::weak_two}. Then
  \[
    \Ltwonorm{\uu - \uh}{} = \Vert \Res(\uh) \Vert_{\vel; \, \Hbplusdual{\dom}} \qquad \forall \uh \in \Ltwo(\dom).
  \]
\end{thm}

\begin{proof}
  The well-posedness of the weak formulation~\eqref{Eq::weak_one}, for the velocity field  $-\vel$, implies that for all $\vv \in \Ltwo(\dom)$, there exists a unique $\zz \in \Hbplus{\dom}$ such that
  \begin{equation*}
    -\integprod{\inprod{\vel}{\grad \zz}}{\ww}{}=\integprod{\vv}{\ww}{} \qquad \forall \ww \in \Ltwo(\dom).
  \end{equation*}
  This clearly gives $\Ltwonorm{\inprod{\vel}{\grad \zz}}{} = \Ltwonorm{\vv}{}$. Hence, for any $\ww \in \Ltwo(\dom)$, we have
  \begin{align*}
    \Ltwonorm{{\ww}}{} = \sup_{\vv \in \Ltwo(\dom) \setminus \{ 0 \} } \dfrac{\integprod{ {\ww}}{ \vv}{}}{\Ltwonorm{\vv}{}} = \sup_{\zz \in \Hbplus{\dom} \setminus \{0\}} \dfrac{ -\integprod{ {\ww}}{ \inprod{\vel}{\grad \zz}}{}}{\Ltwonorm{\inprod{\vel}{\grad \zz}}{}},
  \end{align*}
  and the claim follows by the choice $\ww = \uu - \uh$ and using the definitions~\eqref{Eq::weak_two} and~\eqref{Eq::Res_two}.
\end{proof}
Compared to the similar  equivalence provided in~\cite[Theorem 3.3]{houston1999posteriori}, Theorem~\ref{Thm::Err_Res} shows a form of equality  which declares the optimality of the chosen spaces and  norms.
This is advantageous for the sharpness of the a posteriori error estimation.

\section{Examples of numerical methods}
\label{Sec::Example}
%

Let $\triag = \{ \el \}$ be a mesh of $\dom$, \ie, a division of the one-dimensional domain $\dom$ into non-overlapping intervals covering $\dom$, shape regular in the sense that two neighboring intervals are of comparable size, up to a constant $\kappa_{\triag}$. Let us denote $\hk \defeq  \mathrm{diam}(\el)$ and $\h \defeq \max_{\el \in \triag} \hk$.
We also denote by $\sklt \defeq \cup_{\el \in \triag} \partial \el$  the skeleton of the triangulation $\triag$, coinciding with the set of mesh vertices $\nodh$ in the present one-dimensional case. Moreover, we need to consider the decompositions $\sklt = \sklt^{\mathrm{int}} \cup \sklt^{\mathrm{bnd}}$ into internal and boundary faces and $\nodh = \nodhi \cup \nodhin \cup \nodhout$ into internal, inflow, and outflow vertices, so that in the one-dimensional case $\sklt = \nodh$ and $\sklt^{\mathrm{bnd}} =  \nodhin \cup \nodhout$.
Let $\polp{\poldeg}(\triag)$ denote piecewise polynomial functions of at most degree $\poldeg$ on the mesh $\triag$. The following three numerical methods are classical examples of discretizations of~\eqref{Eq::adv_multiD}. Please note that in Examples~\ref{Exam::PG} and \ref{Exam::DG}, we exclude the lowest polynomial degrees. We need to do so to comply with the orthogonality condition in Assumption~\ref{Ass::hat_orthog}, see Lemma~\ref{Lem::hat orth} below.

The first finite element scheme is a finite-dimensional version of the weak formulation~\eqref{Eq::weak_one}:
\begin{exam}[continuous trial Petrov--Galerkin (\CPG) finite element]
  \label{Exam::PG}
  Find $\uh \in \Xh \defeq \Hbminus{\dom} \cap \polp{\poldeg}(\triag)$, $\poldeg \geq 2$, such that
  \begin{equation}
    \label{Eq::PG}
    \integprod{\inprod{\vel}{\grad \uh}}{\vh}{} = \integprod{\ff}{\vh}{} \qquad \forall \vh \in \Yh \defeq \polp{\poldeg-1}(\triag).
  \end{equation}
\end{exam}

The second finite element scheme stems from the ultra-weak formulation~\eqref{Eq::weak_two}:

\begin{exam}[discontinuous trial Petrov--Galerkin  (\DPG) finite element]

  \label{Exam::DPG}
 Find $\uh \in \Xh \defeq \polp{\poldeg}(\triag)$, $\poldeg \geq 0$, such that
  \begin{equation}
    \label{Eq::DPG}
    -\integprod{\uh}{\inprod{\vel}{\grad \vh}}{} = \integprod{\ff}{\vh}{} \qquad \forall \vh \in \Yh \defeq \Hbplus{\dom} \cap \polp{\poldeg+1}(\triag).
  \end{equation}
\end{exam}

Finally, the {\DG} method for problem~\eqref{Eq::adv_multiD} (letting $\grad$ also denote the broken (elementwise) gradient) reads:

\begin{exam}[{\DG} finite element]
  \label{Exam::DG}
 Find $\uh \in \Xh \defeq \polp{\poldeg}(\triag)$, $\poldeg \geq 1$, such that
  \begin{subequations} \label{Eq::DG}
    \begin{equation}
    \B_\h(\uh, \vh) = (\ff,  \vh) \qquad \forall \vh \in \Yh \defeq \polp{\poldeg}(\triag),
  \end{equation}
  where
  \begin{align}
    \B_\h(\uh, \vh) & \defeq - \sum_{\el \in \triag} \integprod{\uh}{\inprod{\vel}{\grad \vh}}{\el} \nonumber                                                                                                                                                                                            \\
                    & \quad - \sum_{\edg \in \sklt^{\mathrm{int}}} \inprod{\vel}{\nn} \avg{\uh} \jump{\vh} +  \sum_{\edg \in \sklt^{\mathrm{int}}} \frac{1}{2}  \vert \inprod{\vel}{\nn} \vert \jump{\uh} \jump{\vh} + \sum_{\edg \in \sklt^{\mathrm{bnd}}} \lf( \inprod{\vel}{\nn} \rt)^+ \uh \vh.
  \end{align}
  \end{subequations}
  Here the notation $\uh^-$ and $\uh^+$ stands for the trace value on a vertex from left and from right, respectively, the average is defined as $\avg{\uh} \defeq (\uh^- + \uh^+)/2$, and the jump is defined as $\jump{\uh} \defeq \uh^+ - \uh^-$.
  In this formulation, the upwind {\DG} flux is applied on the cell interfaces.
\end{exam}

\section{Error localization}
\label{Sec::localization}
%
In this section, we show that if the numerical solution $\uh$ satisfies a first-order orthogonality condition with respect to hat basis functions, one can obtain a two-sided bound on $\Vert \Res(\uh) \Vert_{\vel; \, \Hbplusdual{\dom}}$ by identifying some (infinite-dimensional) problems on patches of elements around vertices.

\subsection{Patches and partition of unity by the hat functions}

Let $\trinod$ denote the {\em patch} of all simplices which share the given vertex $\nod$, $\trinod \defeq \{ \el, \nod \in \nodel \}$.  Let $\patch$ be the corresponding open subdomain. Then $\cup_{\nod \in \nodh} \patch$ forms an overlapping partition of $\dom$, with $\card = 2$ maximal overlap in one space dimension. For all $\nod \in \nodh$, let $\psia \in \Hs{1}(\dom) \cap \polp{1}(\triag)$ be the piecewise affine hat function, taking value $1$ in vertex $\nod$ and $0$ in all other vertices. The hat functions verify $\text{supp}(\psia) = \overline{\patch}$ and form a {\em partition of unity} as
\begin{equation}
  \label{Eq::unity_partition}
  \sum_{\nod \in \nodh} \psia = 1.
\end{equation}
%


\subsection{Cut-off estimates}
Similarly to~\eqref{eq_Hpm}, let $\Hbplus{\patch}$ contain those functions from $\Hs{1}(\patch)$ with zero trace on the outflow boundary of $\patch$.  Define two patchwise spaces
\begin{align}
  \label{Eq::Va}
  \Ya \defeq \begin{cases}
    \Hs{1}_0(\patch), & \nod \notin  \nodh^{{\bdomm}}, \\
    \Hbplus{\patch},  & \nod \in  \nodh^{{\bdomm}},
  \end{cases}
\end{align}
and
\begin{align}
\label{Eq::tVa}
\tYa(\patch) \defeq \begin{cases}
      \{\Hs{1}(\patch): (\vv, 1)_\patch = 0\}, & \nod \notin  \nodh^{{\bdomp}}, \\
      \Hbplus{\patch},                         & \nod \in  \nodh^{{\bdomp}}.
\end{cases}
\end{align}
In the sequel, we will use several times the following fact:
\begin{equation}\label{eq_impl}
  \vv \in  \tYa(\patch) \Longrightarrow \psia \vv \in \Ya.
\end{equation}
Let us define
  \begin{equation*}
    \cregl \defeq \max_{\nod \in \nodh} \lf(1+ \cpf{\patch} \h_\patch \Vert \grad \psia \Vert_{\infty}  \rt).
  \end{equation*}
We notice that this constant only depends on the shape-regularity constant $\kappa_{\triag}$. As in the present one-dimensional setting, $\vel$ is a constant scalar, the following cut-off Poincar{\'e} estimate follows immediately from~\cite[Theorem~3.1]{Cars_Funk_full_rel_FEM_00} or
\cite[Section~3]{Brae_Pill_Sch_p_rob_09}, \cf\ also~\cite[Lemma~3.12]{ern2015polynomial}.

\begin{lem}[local cut-off estimate]
  \label{Lem::cutoff_PF} For any $\nod \in \nodh$, we have
  \[
    \Vert {\inprod{\vel}{\nabla}} (\psia \vv) \Vert_{{\patch}} \leq \cregl \Vert {\inprod{\vel}{\nabla}} \vv \Vert_{{\patch}} \qquad \forall \vv \in \tYa{(\patch)}.
  \]
\end{lem}
\subsection{Error localization}
The following assumption on the \emph{$\psia$-orthogonality} of the residual will be crucial to localize the error:
\begin{asm}[$\psia$-orthogonality]
  \label{Ass::hat_orthog}
  The residual $\Res(\uh) \in \Hbplusdual{\dom} $ defined in~\eqref{Eq::Res_two} satisfies
  \begin{equation} \label{Eq::hat_orthog}
    \pair{\Res(\uh)}{\psia} = \integprod{\ff}{\psia}{\patch} + \integprod{\uh}{\inprod{\vel}{\grad \psia}}{\patch} = 0  \qquad \forall \nod \in \nodhi \cup \nodh^{{\bdomm}}.
  \end{equation}
\end{asm}
Since the zero-extension of a function in $\Ya$ is in $\Hbplus{\dom}$, we can define the restriction of $\Res(\uh)$ from~\eqref{Eq::Res_two} to the space $\Ya$ as
\begin{equation*}
  \Vert \Res(\uh) \Vert_{\Yap} \defeq \sup_{\vv \in \Ya{\setminus \{0\}}} \dfrac{ \pair{\Res(\uh)}{\vv}}{\Vert \inprod{\vel}{ \grad \vv} \Vert_\patch}.
\end{equation*}
We then have:

\begin{thm}[localizability of residual dual norms  with $\psia$-orthogonality]
  \label{Thm::Localization}
  Provided $\Res(\uh)$ satisfies Assumption~\ref{Ass::hat_orthog}, we have
  \begin{subequations}\begin{equation}
    \label{Eq::Res_loc_upper_bound}
    \Vert \Res{(\uh)} \Vert^2_{{\vel; }\, \Hbplusdual{\dom}}  \leq 2 \cregl^2  \sum_{\nod \in \nodh} \Vert \Res{(\uh)} \Vert^2_{\Yap}.
  \end{equation}
  Independently of Assumption~\ref{Ass::hat_orthog}, the following always holds true:
  \begin{equation}
    \label{Eq::Res_loc_lower_bound}
    \sum_{\nod \in \nodh} \Vert \Res{(\uh)} \Vert^2_{\Yap} \leq 2 \Vert \Res{(\uh)} \Vert^2_{\vel; \, \Hbplusdual{\dom}}.
  \end{equation} \end{subequations}
\end{thm}

\begin{proof}
  The proof proceeds along the lines in~\cite{blechta2016localization, Cars_Funk_full_rel_FEM_00, Brae_Pill_Sch_p_rob_09, ern2015polynomial}. In particular, noting the partition of unity property~\eqref{Eq::unity_partition} and the $\psia$-orthogonality of Assumption~\ref{Ass::hat_orthog}, one can use
  $v = \sum_{\nod \in \nodh} \psia v$ as the test function to obtain, for each $\vv \in \Hbplus{\dom}$
  \begin{align*}
    \pair{\Res(\uh)}{\vv} & \overset{\eqref{Eq::unity_partition}, \eqref{Eq::hat_orthog}}{=}  \sum_{\nod \in \nodhi \cup \nodhin}  {\pair{\Res(\uh)}{\psia(\vv -  \bar \vv_{{\nod}})}} + \sum_{\nod \in \nodhout}  {\pair{\Res(\uh)}{\psia \vv}},
  \end{align*}
  where $\bar \vv_{\nod}$ is the mean value of $\vv$ on $\patch$.
  Let $\ww_\nod \defeq \vv -  \bar \vv_{{\nod}} \vert_\patch$ if $\nod \in \nodhi \cup \nodhin$ and  $\ww_\nod \defeq \vv \vert_\patch$ if $\nod \in \nodhout$. Then, $\ww_\nod \in \tYa(\patch)$, so that $\psia \ww_\nod \in \Ya$ by~\eqref{eq_impl}.
  Using the cut-off estimate of Lemma~\ref{Lem::cutoff_PF} for $\vv = \ww_\nod$, one can in particular obtain~\eqref{Eq::Res_loc_upper_bound} and~\eqref{Eq::Res_loc_lower_bound} like in~\cite[Theorem~3.7]{blechta2016localization}.
\end{proof}

\subsection{\texorpdfstring{$\psi$\textsubscript{$\bm{a}$}}{}-orthogonality of  the residual for the methods of Section~\ref{Sec::Example}}
The following lemma assesses the validity of  Assumption~\ref{Ass::hat_orthog} for the three methods presented in Section~\ref{Sec::Example}:
\begin{lem}[$\psia$-orthogonality of the residual]
  \label{Lem::hat orth}
  For \CPG\ of Example~\ref{Eq::PG} with $\poldeg \geq 2$, \DPG\ of Example~\ref{Exam::DPG} with $\poldeg \geq 0$, and \DG\ of Example~\ref{Exam::DG} with $\poldeg \geq 1$, Assumption~\ref{Ass::hat_orthog} holds true.
\end{lem}

\begin{proof}
  Let $\nod \in \nodhi \cup \nodhin$. We verify the condition for each method:
  \begin{itemize}
    \item From definition~\eqref{Eq::Res_two}, for the {\CPG} method~\eqref{Eq::PG}, we have
          \begin{align}
            \label{Eq::tmp_exam_PG}
            \pair{\Res(\uh)}{\psia} = & \sum_{\el \in \trinod } {\big\{} (\ff,  \psia)_\el + (\uh, \inprod{\vel}{\grad  \psia})_\el {\big\}}\nonumber \\  \overset{\mathrm{I.B.P.}}{=} & \sum_{\el \in \trinod } {\big\{} (\ff,  \psia)_\el - \integprod{\inprod{\vel}{\grad \uh}}{\psia}{\el} + \integprod{\inprod{\vel}{\nn} \uh }{\psia}{\partial \el} {\big\}}.
          \end{align}
          For all $\nod \in \nodhi$, the jump $\jump{\psia}$ vanishes at the vertex $\nod$ and $\psia = 0$ on the boundary edge of the patch. Hence, since $\uh$ is also continuous in $\nod$ in the {\CPG} method, the last term in~\eqref{Eq::tmp_exam_PG} disappears and one infers that
          \begin{align*}
            \pair{\Res(\uh)}{\psia} =  \integprod{\ff}{\psia}{\dom} - \integprod{\inprod{\vel}{\grad \uh}}{\psia}{\dom} \overset{\eqref{Eq::PG}}{=} 0,
          \end{align*}
          since we assume $\poldeg \geq 2$, so that $\psia \in \Yh$. The same result is valid for $\nod \in \nodhin$ since $\uh = 0$ on the inflow as imposed in the definition of $\Xh$.

    \item From definition~\eqref{Eq::Res_two} and employing the {\DPG} characterization~\eqref{Eq::DPG}, we obtain in a straightforward manner that
          \begin{align*}
            \pair{\Res(\uh)}{\psia} = 0
          \end{align*}
    for all $\poldeg \geq 0$.

    \item For the {\DG} method~\eqref{Eq::DG}, noting that $\lf(\inprod{\vel}{\nn} \rt)^+ = 0$ on the inflow and using the same arguments on the vanishing of the jump $\jump{\psia}$ and some $\poldeg \geq 1$ by assumption, we have for any vertex $\nod \in \nodhi \cup \nodhin$
          \begin{equation*}
            \sum_{\edg \in \sklt^{\mathrm{int}}} \lf\{  \frac{1}{2}  \vert \inprod{\vel}{\nn} \vert \jump{\uh} -  \inprod{\vel}{\nn} \avg{\uh} \rt\} \jump{\psia} + \sum_{\edg \in \sklt^{\mathrm{bnd}}} \lf( \inprod{\vel}{\nn} \rt)^+ \uh \psia = 0.
          \end{equation*}
          Hence, also employing definition~\eqref{Eq::Res_two}, we infer that
          \begin{align*}
            \pair{\Res(\uh)}{\psia} & =  \sum_{\el \in \trinod } {\big\{} (\ff,  \psia)_\el + (\uh, \inprod{\vel}{{\grad} \psia})_\el {\big\}} = 0
          \end{align*}
          for all $\poldeg \geq 1$ which implies $\psia \in \Yh$.
  \end{itemize}

\end{proof}

\section{Local problems on patches}
\label{Sec::Local_recons}
In this section, we present a local reconstruction technique which provides the key ingredient to evaluate our a posteriori error estimator.
\begin{defi}[patchwise problems]
  \label{Defi::local_second}
  Let $\uh \in \Ltwo(\dom)$ satisfy Assumption~\ref{Ass::hat_orthog}. For all vertices $\nod \in \nodh$, let $\sha \in \Xha$ be the solution of the following advection--reaction problem on the patch $\patch$
  %
  \begin{align}
    \label{Eq::Recons}
    \integprod{\inprod{\vel}{\grad(\psia \sha)}}{ \vh}{\patch} = \integprod{\ff \psia + \lf(\inprod{\vel}{\grad \psia} \rt) \uh}{\vh}{\patch} \qquad \forall \vh \in \Yha,
  \end{align}
  %
  %
  with the finite-dimensional spaces
  \begin{align*}
    \Xha \defeq
    \polp{\poldegp}(\trinod) \cap \Hs{1}(\patch),
    \qquad \Yha \defeq  \polp{\poldegp}(\trinod)
  \end{align*}
   and $\poldegp \geq 0$. Define the global reconstruction $\sh$ by
  \begin{equation}\label{eq_sh}
    \sh \defeq \sum_{\nod \in \nodh} \psia \sha.
  \end{equation}
\end{defi}

\begin{remark}[trial and test spaces]
  \label{Rem::local}
  A priori, the number of degrees of freedom in $\Xha$ and $\Yha$ for $\nod \in \nodhi$ does not match; while there exist $2(\poldegp + 1)$ linearly independent test functions in $\Yha$, the trial space $\Xha$ has only $2 \poldegp + 1$ degrees of freedom.
  For any $\nod  \in \nodhi$, though, the test function in~\eqref{Eq::Recons} given by $\vh = 1$ on both $\el \in \trinod$ is actually superfluous. Indeed, on the one hand, we have
  \begin{equation}
    \label{Eq::tmp_remark_additional}
    \integprod{\inprod{\vel}{\grad(\psia \sha)}}{ 1}{\patch} = \integprod{\inprod{\vel}{\nn}}{ \psia \sha}{\partial \patch} = 0,
  \end{equation}
  according to  the definition of $\psia$. On the other hand, Assumption~\ref{Ass::hat_orthog} guarantees that the right-hand side vanishes in such a case, hence
  \begin{equation*}
    \integprod{\ff \psia + \lf(\inprod{\vel}{\grad \psia}\rt) \uh }{1}{\patch} =  \pair{\Res(\uh)}{\psia} = 0.
  \end{equation*}
  %
\end{remark}

We next show that the solution of~\eqref{Eq::Recons} uniquely exists and the proposed reconstruction is well-posed:

\begin{lem}[well-posedness of Definition~\ref{Defi::local_second}]
  \label{Lem::WP}
  There exists a unique solution $\sha \in \Xha$ of problem~\eqref{Eq::Recons}. It is stable in the sense that
  \[
    \Ltwonorm{\sha}{\patch} \leq C \lf( \h_\patch\Ltwonorm{ \ff}{\patch} +  \Ltwonorm{ \uh}{\patch} \rt),
  \]
  for some constant $C > 0$ only depending on the shape-regularity constant $\kappa_{\triag}$, the polynomial degree $\poldegp$, and the advection parameter $\vel$.
\end{lem}

\begin{proof}
  Since $\psia {\sha} \in {\Hs{1}_0(\patch)}$ for all $\nod \in \nodhi$, $\psia {\sha} \in \Hbplus{\patch}$ for $\nod \in \nodhin$, and $\psia {\sha} \in \Hbminus{\patch}$ for $\nod \in \nodhout$, $\Vert \inprod{\vel}{\grad(\psia \cdot)} \Vert_{\patch}$ is a norm on $\Xha$. Noting that  $\vel$ is constant and $\inprod{\vel}{\grad(\psia \sha)} \in \polp{\poldegp}(\trinod) = \Yha$, one can write the inf--sup condition of the  bilinear form associated with the left-hand side of~\eqref{Eq::Recons} as
  \begin{align*}
    \sup_{\vh \in \Yha {\setminus \{0\}}} \dfrac{\integprod{\inprod{\vel}{\grad(\psia \sha)}}{\vh}{\patch}}{ \Vert \vh \Vert_{\patch} } = \Vert \inprod{\vel}{\grad(\psia \sha)} \Vert_{\patch},
  \end{align*}
  with unit inf--sup constant. Following Remark~\ref{Rem::local}, this injectivity  implies the bijectivity of the operator.

  To derive a bound on $\sha$, we observe that
  \begin{equation}
    \label{Eq::stab}
    \Vert \inprod{\vel}{\grad ( \psia \sha)} \Vert_{\patch} \leq \Ltwonorm{ \ff \psia + \lf(\inprod{\vel}{\grad \psia} \rt) \uh}{\patch}.
  \end{equation}
  If one wants to check the stability in the $\Ltwo$-norm, one can start with the following norm equivalence on $\Yha$:
  \begin{equation*}
    \Vert \psia \sha \Vert_{\patch} \leq \Vert \sha \Vert_{\patch} \leq C(\poldegp) \Vert \psia \sha \Vert_{\patch},
  \end{equation*}
  using similar arguments as in~\cite[Lemma 3.42]{verfurthbook}. Consequently, one has
  \begin{equation}
    \label{Eq::WP_auxtwo}
    \Vert \inprod{\vel}{\grad(\psia \sha)} \Vert_{\patch} \geq  \dfrac{{1}}{{C(\poldegp)} \cpf{\patch} \h_\patch} \vert \vel \vert \Vert \sha \Vert_{\patch},
  \end{equation}
  employing the Poincar\'e inequality~\eqref{Eq::Friedrichs}.
  Using~\eqref{Eq::WP_auxtwo} and~\eqref{Eq::stab}, we infer that
  \begin{align*}
    \Vert \sha \Vert_{\patch}  \leq \dfrac{{C(\poldegp)} \cpf{\patch} \h_\patch}  { \vert \vel \vert} \Ltwonorm{ \ff \psia + \lf(\inprod{\vel}{\grad \psia} \rt)\uh}{\patch} \ \leq {C(\poldegp)} \cpf{\patch} \lf[ \dfrac{\h_\patch}  { \vert \vel \vert} \Ltwonorm{ \ff }{\patch} + {C(\kappa_{\triag})} \Ltwonorm{ \uh}{\patch} \rt].
  \end{align*}

\end{proof}
The following lemma presents the main properties of the reconstruction $\sh$ from  Definition~\ref{Defi::local_second}:
\begin{lem}[properties of the reconstruction]
  \label{Lemm::orthog}
  Definition~\ref{Defi::local_second} yields $\sh$ such that
  \begin{equation}
    \label{Eq::bnd}
    \sh \in \polp{\poldegp+1}(\triag) \cap \Hbminus{\dom},
  \end{equation}
  \ie, it lies in a natural finite-dimensional functional space corresponding to the weak formulation~\eqref{Eq::weak_one}. Moreover, the following orthogonality is satisfied
  \begin{equation}
    \label{Eq::orthog}
    \integprod{\ff - \inprod{\vel}{\grad \sh}}{ \vh}{\el}  = 0 \qquad  \forall \vh \in \polp{\poldegp}(\el), \quad \forall \el \in \triag.
  \end{equation}
\end{lem}

\begin{proof}
  For~\eqref{Eq::bnd} is clear that $\sh \in \polp{\poldegp+1}(\triag) \cap \Hs{1}(\dom)$, and we only need to show that $\sh$ satisfies the boundary condition requirement of the space $\Hbminus{\dom}$, \ie, $\sh \big \vert_{\bdomm} = 0$. We check this by showing that $\sha\big \vert_{ \partial \patch \cap \bdomm} = 0$ for $\nod \in \nodhin$. We see from~\eqref{Eq::Recons} and Assumption~\ref{Ass::hat_orthog} that
  \[
  \integprod{\inprod{\vel}{\grad(\psia \sha)}}{ 1}{\patch} = \integprod{\ff \psia + \lf(\inprod{\vel}{\grad \psia} \rt) \uh}{1}{\patch} =  \pair{\Res(\uh)}{\psia} = 0,
  \]
  so that the requested equality follows from integration-by-parts similarly to~\eqref{Eq::tmp_remark_additional},
  \[
    \integprod{\inprod{\vel}{\grad(\psia \sha)}}{ 1}{\patch} = \inprod{\vel}{\nn}{\sha} \vert_{\bdomm},
  \]
  and since $\inprod{\vel}{\nn} \neq 0$ on $\bdomm$ by definition.

  To prove~\eqref{Eq::orthog}, first note that $ \sum_{\nod \in \nodel} \psia \big\vert_\el = 1$ and $\sum_{\nod \in \nodel} \lf( \inprod{\vel}{\grad \psia} \rt) \uh \big \vert_\el = 0$. Thus, since $\Yha \vert_\el = \polp{\poldegp}(\el)$, extending the function $\vh {\in \polp{\poldegp}(\el)}$ by zero outside $\el$, and using respectively definitions~\eqref{eq_sh} of $\sh$ and~\eqref{Eq::Recons} of $\sha$, one has
  \begin{align*}
    \integprod{\ff - \inprod{\vel}{\grad \sh}}{ \vh}{\el} & = \integprodBigg{\sum_{\nod \in \nodel} \big\{ \psia \ff  + \lf( \inprod{\vel}{\grad \psia} \rt) \uh - \inprod{\vel}{\grad (\psia \sha}) \big\}}{ \vh}{\el}           \\
    & = \sum_{\nod \in \nodel} \integprod{  \psia \ff  + \lf( \inprod{\vel}{\grad \psia} \rt)  \uh - \inprod{\vel}{\grad (\psia \sha})}{ \vh}{\patch}  = 0.
  \end{align*}
\end{proof}
\begin{rem}[local conservation]
  As  a special case of~\eqref{Eq::orthog}, since $1 \vert_\el \in \polp{\poldegp}(\el)$ for all $\poldegp \geq 0$, one has the conservation property
  \begin{equation}
    \label{Eq::mean_zero}
    \integprod{\ff - \inprod{\vel}{\grad \sh}}{1}{\el} = 0 \qquad \forall \el \in \triag.
  \end{equation}
\end{rem}

\section{Guaranteed a posteriori estimate}
\label{Sec::upper_bound}
The guaranteed upper bound on the error can be presented as follows:
\begin{thm}[guaranteed a posteriori error estimate]
  \label{Thm::Local_second_reli}
  Let   $\uu \in \Ltwo(\dom)$ be the solution of~\eqref{Eq::weak_two} and let $\uh \in \Ltwo(\dom)$ be arbitrary subject to the $\psia$-orthogonality in Assumption~\ref{Ass::hat_orthog}. Furthermore, consider $\sh$ to be the reconstruction from Definition~\ref{Defi::local_second} with $\poldegp \geq 0$. Then
  \begin{equation*}
    \Ltwonorm{\uu - \uh}{} \leq \Err \defeq \lf\{\sumk \lf( \ErrNCel{\el} +  \ErrOscf{\el} \rt)^2 \rt \}^{1/2},
  \end{equation*}
  where
  \begin{equation*}
    \ErrNCel{\el} {\defeq} \Vert \uh - \sh \Vert_\el
  \end{equation*}
  and the data oscillation estimator is given as
  \begin{equation}
    \label{Eq::Err_Res}
    \ErrOscf{\el} {\defeq} \dfrac{\h_\el}{{\pi}\vert \vel \vert} \Vert (I - \proj_{{\polp{\poldegp}(\triag)}}) \ff \Vert_\el,
  \end{equation}
  with $\proj_{{\polp{\poldegp}(\triag)}}$ the $\Ltwo(\dom)$-orthogonal projection onto $\polp{\poldegp}(\triag)$.
\end{thm}

\begin{proof}
  Since $\sh \in \Hbminus{\dom}$ by Lemma~\ref{Lemm::orthog}, for any $\vv \in \Hbplus{\dom}$ the integration-by-parts formula~\eqref{Eq::IBP} implies that
  \begin{equation}\label{eq_IPPS_sh}
    \integprod{\sh}{\inprod{\vel}{\grad \vv}}{}  + \integprod{\inprod{\vel}{\grad \sh}}{\vv}{} = \integprod{\sh}{\vh \inprod{\vel}{\nn}}{} = 0.
  \end{equation}
  By using the error-residual identity of Theorem~\ref{Thm::Err_Res}, definitions~\eqref{Eq::Res_two}--\eqref{Eq::Res_two_norm}, and the above equality, one can write

  \begin{align*}
    \Ltwonorm{\uu - \uh}{\dom} & =  \Vert \Res({\uh)} \Vert_{{\vel; }\, \Hbplusdual{\dom}}  = \sup_{\vv \in  \Hbplus{\dom} {\setminus \{0\}}}  \dfrac{(\ff - \inprod{ \vel}{ \grad \sh}, \vv)  + (\uh - \sh, \inprod{\vel}{\grad \vv})}{\Ltwonorm{\inprod{\vel}{\grad \vv}}{}}.
  \end{align*}
  Owing to~\eqref{Eq::mean_zero}, denoting by $\bar \vv_{{\el}}$ the mean value of $\vv$ over the element $\el$, we infer that
  \begin{align*}
    \Ltwonorm{\uu - \uh}{}  \overset{\eqref{Eq::mean_zero}}{{=}} & \sup_{\vv \in  \Hbplus{\dom} {\setminus \{0\}}}  \dfrac{ \displaystyle\sumk \Big[ \integprod{\uh - \sh}{ \inprod{\vel}{\grad \vv}}{\el} +  \integprod{\ff - \inprod{ \vel}{ \grad \sh}}{\vv - \bar \vv_{{\el}}}{\el} \Big]}{\Ltwonorm{\inprod{\vel}{\grad \vv}}{}} \nonumber                                                                                \\
    \overset{{\eqref{Eq::Poincare}}}{\leq} & \sup_{\vv \in  \Hbplus{\dom} {\setminus \{0\}}} \dfrac{ \displaystyle\sumk \Big[ \Ltwonorm{\uh - \sh}{\el} \Ltwonorm{\inprod{\vel}{\grad \vv}}{\el} + \frac{\hk}{{\pi}\vert \vel \vert}  \Ltwonorm{\ff -  \inprod{ \vel}{ \grad \sh}}{\el} \Ltwonorm{\inprod{\vel}{\nabla \vv}}{\el} \Big] }{\Ltwonorm{\inprod{\vel}{\grad \vv}}{}}  \nonumber \\
    \leq                                                          & \lf \{ \sumk \lf[ \Ltwonorm{\uh - \sh}{\el} + \frac{\h_\el}{{\pi}\vert \vel \vert} \Ltwonorm{\ff -  \inprod{ \vel}{ \grad \sh}}{\el} \rt]^2 \rt \}^{1/2}.
  \end{align*}
  Noting that $\inprod{\vel}{\grad \sh} \in \polp{\poldegp}(\triag)$, it follows from~\eqref{Eq::orthog} that $\inprod{\vel}{\grad \sh} = \proj_{{\polp{\poldegp}(\triag)}} \ff$ so that
  \begin{align*}
    \Ltwonorm{\ff -  \inprod{ \vel}{ \grad \sh}}{\el} & =  \Vert (I - \proj_{{\polp{\poldegp}(\triag)}}) \ff \Vert_\el,
  \end{align*}
  which completes the proof.
\end{proof}

\begin{rem}[data oscillation] We call the estimator~\eqref{Eq::Err_Res}  ``data oscillation'' for the following reason: if $\uh$ is piecewise polynomial of degree $\poldeg {\geq 0}$, the error $\Ltwonorm{\uu - \uh}{}$ may converge as $\mathcal{O}(\h^{\poldeg+1})$. By choosing $\poldegp\geq \poldeg$ one obtains, for sufficiently piecewise smooth data $\ff$, the higher convergence order $\mathcal{O}(\h^{{\poldegp}+2})$ for these terms.
\end{rem}

\section{Efficiency and robustness}
\label{Sec::efficiency}
In this section, we show that the error estimate introduced in Theorem~\ref{Thm::Local_second_reli} also gives, up to a constant and up to the data oscillation, a lower bound on the  error $\Ltwonorm{\uu - \uh}{}$. Furthermore, the involved constants are independent of the polynomial degree $\poldeg$ and the velocity $\vel$. Actually, a local efficiency result also holds true, and we start with it.

\subsection{Local efficiency and robustness with respect to advection and polynomial degree}
Our main theorem on local efficiency and robustness is:
\begin{thm}[local efficiency and robustness]
  \label{Thm::Local_second_eff}
  Let $\uu \in \Ltwo(\dom)$ be the weak solution of~\eqref{Eq::weak_two} and let $\uh \in \polp{\poldeg}(\triag)$, $\poldeg \geq 0$, be its approximation. Consider $\sh$ as obtained by Definition~\ref{Defi::local_second} with $\poldegp \geq \poldeg$ and $\ErrNCel{\el}$ as defined in Theorem~\ref{Thm::Local_second_reli}.
  Then, for all $\el \in \triag$, the following holds true
  \begin{align*}
    \ErrNCel{\el} & \leq \cregl  \sum_{\nod \in \nodel}  \Vert \uu - \uh \Vert_\patch + \sum_{\nod \in \nodel} \dfrac{\h_\patch}{{\pi} \vert \vel \vert} \Vert (I - \proj_{{\polp{\poldegp}(\trinod)}}) (\ff \psia) \Vert_\patch.
  \end{align*}
\end{thm}

\begin{proof}
  Fix an element $\el \in \triag$. Noting that $\sum_{\nod \in \nodel} \psia \big \vert_\el  = 1$ and using definition~\eqref{eq_sh}, one has
  \begin{equation} \label{eq_el_patch}
    \Vert \uh - \sh \Vert_\el  =  \Bigg\Vert \sum_{\nod \in \nodel} \psia {(} \uh -  \sha {)} \Bigg\Vert_\el \leq   \sum_{\nod \in \nodel} \Vert \psia(\uh - \sha) \Vert_\patch.
  \end{equation}
  Recalling~\eqref{Eq::tVa}, we easily see that for any  vertex $\nod \in \nodh$, there is a unique $\vv^{{\nod}} \in \tYa(\patch)$ such that
  \[
    \inprod{\vel}{\grad \vv^{{\nod}}} =  \psia(\uh - \sha)
  \]
  in $\patch$, and $\vv^{{\nod}}$ is nonzero unless $\psia\uh = \psia\sha$, in which case $\Vert \psia(\uh - \sha) \Vert_\patch =0$. Moreover, first, $\lf(\psia \sha\rt) (\nod) = \sh (\nod) = 0$ when $\nod \in  \nodh^{{\bdomm}}$, using~\eqref{Eq::bnd}, and, second, $\vv^{{\nod}}(\nod) =0$ when $\nod \in  \nodh^{{\bdomp}}$, using~\eqref{Eq::tVa}. Thus, similarly to~\eqref{eq_IPPS_sh}, for any $\nod \in \nodh$, we have
    \[
    \integprod{\psia \sha}{\inprod{\vel}{\grad \vv^{{\nod}}}}{\patch}  + \integprod{\inprod{\vel}{\grad (\psia \sha)}}{\vv^{{\nod}}}{\patch} = 0.
  \]
  From the two above identities, we infer that
  \begin{equation}\label{Eq::tmp_eff_bis}\begin{split}
    \Vert \psia(\uh - \sha) \Vert_\patch & = \dfrac{\integprod{\psia(\uh - \sha)}{ \inprod{\vel}{\grad \vv^{{\nod}}}}{\patch}}{\Vert \inprod{\vel}{\grad \vv^{{\nod}}} \Vert_\patch} \\
    & = \dfrac{\integprod{\psia\uh}{ \inprod{\vel}{\grad \vv^{{\nod}}}}{\patch} + \integprod{\ff \psia + \inprod{\vel}{\grad \psia} \uh}{ \vv^{{\nod}}}{\patch}}{\Vert \inprod{\vel}{\grad \vv^{{\nod}}} \Vert_\patch}  \\
    & \quad +  \dfrac{\integprod{\inprod{\vel}{\nabla(\psia \sha)}}{\vv^{{\nod}}}{\patch} - \integprod{\ff \psia + \inprod{\vel}{\grad \psia} \uh}{ \vv^{{\nod}}}{\patch} }{\Vert \inprod{\vel}{\grad \vv^{{\nod}}} \Vert_\patch} \\
    & \eqdef I + II.
  \end{split}\end{equation}

  For the term $I$, remark first that from~\eqref{eq_impl} and from the definition of $\Ya$ in~\eqref{Eq::Va}, we have
  \begin{equation*}
    \psia \vv^{{\nod}} \in \Ya \subseteq \Hbplus{\patch} \subseteq \Hbplus{\dom}.
  \end{equation*}
  Second, recalling the residual definition~\eqref{Eq::Res_two} and the ultra-weak formulation~\eqref{Eq::weak_two}, we have
  \[
    \integprod{\ff \psia + \inprod{\vel}{\grad \psia} \uh}{ \vv^{{\nod}}}{\patch} + \integprod{ \uh \psia}{\inprod{\vel}{\grad \vv^{{\nod}}}}{\patch} = \pair{\Res(\uh)}{\psia \vv^{{\nod}}} {= - \integprod{\uu - \uh}{\inprod{\vel}{\grad (\psia \vv^{{\nod}})}}{}}.
  \]
  Consequently, employing the Cauchy--Schwarz inequality and Lemma~\ref{Lem::cutoff_PF}, we infer that
  \begin{equation} \label{Eq::tmp_eff}
    I = \dfrac{{- \integprod{\uu - \uh}{\inprod{\vel}{\grad (\psia \vv^{{\nod}})}}{}} }{\Vert \inprod{\vel}{\grad (\psia \vv^{{\nod}})} \Vert_\patch} \dfrac{\Vert \inprod{\vel}{\grad (\psia \vv^{{\nod}})} \Vert_\patch}{\Vert \inprod{\vel}{\grad \vv^{{\nod}}} \Vert_\patch}  \leq \cregl {\Vert \uu - \uh \Vert_\patch}.
  \end{equation}

  To bound the term $II$, we use the fact that $\lf(\inprod{\vel}{\grad \psia}\rt) \uh \in \Yha$ when $\poldegp \geq \poldeg$ and that $\inprod{\vel}{\grad {(\psia} \sha {)}} \in \Yha$, so that~\eqref{Eq::Recons} actually holds pointwise, in the form
  \begin{equation*}
    \inprod{\vel}{\grad \lf(\psia {\sha} \rt) } = \proj_{{\polp{\poldegp}(\trinod)}}(\ff \psia) + \lf(\inprod{\vel}{\grad \psia}\rt) \uh.
  \end{equation*}
  Hence, denoting $\bar \vv^{{\nod}}_{{\el}}$ the mean value of $\vv^{{\nod}}$ over the element $\el \in \trinod$ and using~\eqref{Eq::Poincare}, we obtain
  \begin{align*}
    II & = \dfrac{(\proj_{{\polp{\poldegp}(\trinod)}}(\ff \psia) - \ff \psia , \vv^{{\nod}})_{{\patch}} }{\Vert \inprod{\vel}{\grad \vv^{{\nod}}} \Vert_\patch} {= \dfrac{\sum_{\el \in \trinod}(\proj_{{\polp{\poldegp}(\trinod)}}(\ff \psia) - \ff \psia , \vv^{{\nod}} {- \bar \vv^{{\nod}}_\el})_{\el} }{\Vert \inprod{\vel}{\grad \vv^{{\nod}}} \Vert_\patch}}\\
    & \leq \dfrac{\max_{\el \in \trinod}\h_\el}{{\pi}\vert \vel \vert}  \Vert \proj_{{\polp{\poldegp}(\trinod)}}(\ff \psia) - \ff \psia \Vert_\patch.
  \end{align*}
  The assertion follows by combining the bounds on $I$ and $II$ with~\eqref{eq_el_patch}.
\end{proof}

\subsection{Global efficiency and maximal overestimation}
In this section, we show a result on maximal global overestimation, leaving out the data oscillation term for simplicity.

\begin{lem}[global efficiency and maximal overestimation]
  \label{Lem::Maximal}
  Let the assumptions of Theorem~\ref{Thm::Local_second_eff} be valid and assume in addition that $\psia \ff \in \Yha$ for all $\nod \in \nodh$. Then
  \[
    \Ltwonorm{\uh - \sh}{} \leq 2 \cregl \Ltwonorm{ \uu - \uh }{}.
  \]
\end{lem}

\begin{proof} Proceeding as in the proof of Theorem~\ref{Thm::Local_second_eff}, one has
\begin{align*}
    \Ltwonorm{\uh - \sh}{}^2 & = \sumk  \LtwonormBigg{ \sum_{\nod \in \nodel} \psia \lf( \uh - \sha \rt)}{\el}^2  \leq  2 \sumk  \sum_{\nod \in \nodel}  \Ltwonorm{ \psia \lf( \uh - \sha \rt)}{\el}^2       \\
    & = 2 \sum_{\nod \in \nodh}  \Ltwonorm{ \psia \lf( \uh - \sha \rt)}{\patch}^2   \overset{\eqref{Eq::tmp_eff_bis},\eqref{Eq::tmp_eff}}{\leq}   2     \cregl^2  \sum_{\nod \in \nodh} \Vert \uu - \uh \Vert^2_\patch.
\end{align*}
Another estimate for the overlapping of the patches yields $\sum_{\nod \in \nodh} \Ltwonorm{{\uu - \uh}}{\patch}^2 \leq 2 \Ltwonorm{{\uu - \uh}}{}^2$ and leads to the assertion.
\end{proof}

\section{Numerical experiments}
\label{sec:experiments}
We provide in this section a numerical illustration of our results in one space dimension. In the first set of examples in Section~\ref{Sec::Exp_vel}, we consider a polynomial right-hand side function $\ff$ and study the robustness of our estimators with respect to the velocity field  $\vel$. Then, in Section~\ref{Sec::Exp_deg}, we consider a more general case to investigate the effect of the increase of the polynomial degree  on the quality of the estimators. Henceforth, we consider $\dom = (0, 1)$ with the mesh $\triag = \{ \el_i \}_{{1}}^n$ with $\el_i =  [x_{i-1}, x_i ]$.
In the experiments, the numerical solution $\uh \in \polp{\poldeg}(\triag)$ will be computed by two methods:
\begin{itemize}
  \item[-] the {\DPG} method~\eqref{Eq::DPG} with the finite-dimensional spaces as in Example~\ref{Exam::DPG}, $\poldeg \geq 0$,

  \item[-] the {\DG} method~\eqref{Eq::DG} with the finite-dimensional spaces as in Example~\ref{Exam::DG}, $\poldeg \geq 1$.
\end{itemize}
The \emph{effectivity index} is defined as
$
  I_{\mathrm{eff}} \defeq \dfrac{\Err}{\Ltwonorm{\uu - \uh}{}},
$
\ie, as the ratio of the estimated and the actual error from Theorem~\ref{Thm::Local_second_reli}.

\subsection{Robustness with respect to the velocity}
\label{Sec::Exp_vel}
Here we consider the advection problem~\eqref{Eq::adv_multiD} with the piecewise quadratic right-hand side defined as
\begin{equation*}
  f(x) =  x^2 + x + \sin(2 \pi x_{i-1}), \qquad \text{on } \el_i {, \, 1 \leq i \leq n,}
\end{equation*}
whose exact solution can be easily computed by integration of the right-hand side. The numerical solutions $\uh$ are obtained by both {\DPG} and {\DG} methods with $\poldeg =1,2$.

If one sets $\poldegp= 2$ in Definition~\ref{Defi::local_second}, the oscillation estimators $\ErrOscf{\el}$ from~\eqref{Eq::Err_Res} disappear. In this case, actually, since $\ff \in \polp{2}({\triag})$, one has $\sh \in \polp{3}({\triag}) \cap {\Hbminus{\dom}}$, see~\eqref{Eq::bnd}. Moreover, owing to~\eqref{Eq::orthog},
$
  \inprod{\vel}{\grad \sh} = \ff
$
pointwise. Hence, $\sh$ in this setting coincides with the exact solution $\uu$ and $I_\mathrm{eff} = 1$ (up to the machine precision), which is numerically confirmed in Tables~\ref{Tab::robust_DPG} and~\ref{Tab::robust_DG}.

To asses the behavior in the case where the reconstruction $\sh$ does not coincide with the exact solution, we also test the choice $\poldegp= 1$ in Definition~\ref{Defi::local_second} together with $\poldeg = 1$. We observe in Tables~\ref{Tab::robust_DPG} and \ref{Tab::robust_DG} that there is still no dependency of the efficiency of our estimates on the magnitude of the velocity $\vel$, in confirmation of the theory. Actually, one remarks that solely scaling $\vel$ in~\eqref{Eq::adv_multiD} by a factor implies the same scaling of all the exact solution $\uu$, the numerical approximations $\uh$, the error $\Vert \uu - \uh \Vert$, the reconstruction $\sh$, and of all the estimators in Theorem~\ref{Thm::Local_second_reli} by the inverse of this factor, so that the effectivity indices actually remain intact here on each given mesh.
Moreover, we numerically observe asymptotic exactness with mesh refinement, for both schemes tested.}

\begin{table}[H]
  \centering
  \caption{Effectivity indices $I_{\mathrm{eff}}$ for different values of the velocity  $\vel$ and $\uh$ obtained by  the {\DPG} method~\eqref{Eq::DPG}}
  \begin{tabular}{|c|c|*{4}{c} c |}
    \hline
    \multicolumn{2}{|c|}{\bm{$\poldeg = \poldegp = 2$} } & \multicolumn{5}{c|}{$\vel$} \\ \hline
    \# Elements & DOF($\uh$) & \makebox[3em]{$10^{-4}$} & \makebox[3em]{$10^{-2}$} & \makebox[3em]{$10^{0}$} & \makebox[3em]{$10^{2}$} & \makebox[3em]{$10^{4}$} \\  \hline

    4           & 12         & 1.0                      & 1.0                      & 1.0                     & 1.0                     & 1.0
    \\ 16 & 48 & 1.0 & 1.0 & 1.0 & 1.0 & 1.0
    \\ 64 & 192 & 1.0 & 1.0 & 1.0 & 1.0 & 1.0
    \\ 256 & 768 & 1.0 & 1.0 & 1.0 & 1.0 & 1.0

    \\
    \hline
    \hline
    \multicolumn{2}{|c|}{\bm{$\poldeg = \poldegp = 1$} } & \multicolumn{5}{c|}{$\vel$} \\ \hline
    \# Elements & DOF($\uh$) & \makebox[3em]{$10^{-4}$} & \makebox[3em]{$10^{-2}$} & \makebox[3em]{$10^{0}$} & \makebox[3em]{$10^{2}$} & \makebox[3em]{$10^{4}$} \\  \hline

    4           & 8          & 1.234                    & 1.234                    & 1.234                   & 1.234                   & 1.234
    \\ 16 & 32 & 1.058 & 1.058 & 1.058 & 1.058 & 1.058
    \\ 64 & 128 & 1.014 & 1.014 & 1.014 & 1.014 & 1.014
    \\ 256 & 512 & 1.004 & 1.004 & 1.004 & 1.004 & 1.004

    \\
    \hline
  \end{tabular}
  \label{Tab::robust_DPG}
\end{table}

\begin{table}[H]
  \centering
  \caption{Effectivity indices $I_{\mathrm{eff}}$ for different values of the velocity  $\vel$ and $\uh$ obtained by the {\DG} method~\eqref{Eq::DG}}
  \begin{tabular}{|c|c|*{4}{c} c |}
    \hline
    \multicolumn{2}{|c|}{\bm{$\poldeg = \poldegp = 2$} } & \multicolumn{5}{c|}{$\vel$} \\ \hline
    \# Elements & DOF($\uh$) & \makebox[3em]{$10^{-4}$} & \makebox[3em]{$10^{-2}$} & \makebox[3em]{$10^{0}$} & \makebox[3em]{$10^{2}$} & \makebox[3em]{$10^{4}$} \\  \hline

    4           & 12         & 1.0                      & 1.0                      & 1.0                     & 1.0                     & 1.0
    \\ 16 & 48 & 1.0 & 1.0 & 1.0 & 1.0 & 1.0
    \\ 64 & 192 & 1.0 & 1.0 & 1.0 & 1.0 & 1.0
    \\ 256 & 768 & 1.0 & 1.0 & 1.0 & 1.0 & 1.0

    \\
    \hline
    \hline
    \multicolumn{2}{|c|}{\bm{$\poldeg = \poldegp = 1$} } & \multicolumn{5}{c|}{$\vel$} \\ \hline
    \# Elements & DOF($\uh$) & \makebox[3em]{$10^{-4}$} & \makebox[3em]{$10^{-2}$} & \makebox[3em]{$10^{0}$} & \makebox[3em]{$10^{2}$} & \makebox[3em]{$10^{4}$} \\  \hline

    4           & 8          & 1.126                    & 1.126                    & 1.126                   & 1.126                   & 1.126
    \\ 16 & 32 & 1.032 & 1.032 & 1.032 & 1.032 & 1.032
    \\ 64 & 128 & 1.008 & 1.008 & 1.008 & 1.008 & 1.008
    \\ 256 & 512 & 1.002 & 1.002 & 1.002 & 1.002 & 1.002

    \\
    \hline
  \end{tabular}
  \label{Tab::robust_DG}
\end{table}

\subsection{Robustness with respect to the polynomial degree}
\label{Sec::Exp_deg}

We now consider the advection problem~\eqref{Eq::adv_multiD} with a non-polynomial right-hand side $\ff(x) = \tan^{-1}(x)$ and $\vel = 1$, for different polynomial degrees $0 \leq \poldeg \leq 4$. The results are presented in Table~\ref{Tab::conv_DPG}
for the {\DPG} method and in Table~\ref{Tab::conv_DG} for the {\DG} method. We always set $\poldegp= \poldeg$.
We use the notation $\ErrNC \defeq \lf( \sumk \ErrNCel{\el}^2 \rt)^{1/2}$ and $\ErrOsc \defeq \lf( \sumk \ErrOscf{\el}^2 \rt)^{1/2}$. The mesh is refined uniformly until the error estimator $\Err \leq 10^{-14}$; we encountered some irregularities in $I_{\mathrm{eff}}$ beyond this point due to machine precision.
We observe optimal convergence order  of the estimators and the independence of $I_{\mathrm{eff}}$ from the polynomial degree, in accordance with the theory.

\begin{table}[H]
  \centering
  \caption{Convergence of the error $\Vert \uu - \uh \Vert$, the error estimators $\Err$, $\ErrNC$, and $\ErrOsc$, and the effectivity indices $I_{\mathrm{eff}}$ for the {\DPG} method~\eqref{Eq::DPG} with different polynomial degrees $k$}
  \begin{tabular}{| c | c   | c |   c  |c |  c |  c |}
    \hline

    \multicolumn{7}{|c|}{$\bm{\poldeg = 0, \poldegp= 0}$ } \\ \hline
    \# Elements & \# DOF($\uh$) & $\ErrNC$   & $\ErrOsc$  & $\Vert \uu - \uh \Vert$ & $\Err$  & $I_\mathrm{eff}$ \\ \hline

      4 & 4 & 3.574e-2 & 1.446e-2 & 3.562e-2 & 4.809e-2 & 1.35\\
      16 & 16 & 8.936e-3 & 9.040e-4 & 8.934e-3 & 9.656e-3 & 1.08\\
      64 & 64 & 2.234e-3 & 5.650e{-5} & 2.234e-3 & 2.278e-4 & 1.02\\
      256 & 256 & 5.585e-4 & 3.531e{-6} & 5.585e-4 & 5.612e-4 & 1.01\\
      1024 & 1024 & 1.396e-4 & 2.207e{-7} & 1.396e-4 & 1.398e-4 & 1.00

    \\ \hline \hline

    \multicolumn{7}{|c|}{\bm{$\poldeg = 1, \poldegp= 1$} } \\ \hline
    \# Elements & \# DOF($\uh$) & $\ErrNC$   & $\ErrOsc$  & $\Vert \uu - \uh \Vert$ & $\Err$  & $I_\mathrm{eff}$ \\ \hline

    4           & 8             & 1.867e-3   & 3.074e-4   & 1.868e-3                & 2.147e-3   & 1.15
    \\ 16 & 32 & 1.167e-4 & 4.811e{-6} & 1.167e-4 & 1.210e-4 & 1.04
    \\ 64 & 128 & 7.294e{-6} & 7.518e{-8} & 7.294e{-6} & 7.361e{-6} & 1.01
    \\ 256 & 512 & 4.559e{-7} & 1.175e{-9} & 4.559e{-7} & 4.569e{-7} & 1.00
    \\ 1024 & 2048 & 2.849e{-8} & 1.836e{-11} & 2.849e{-8} & 2.851e{-8} & 1.00

    \\ \hline \hline

    \multicolumn{7}{|c|}{\bm{$\poldeg = 2, \poldegp= 2$} } \\ \hline
    \# Elements & \# DOF($\uh$) & $\ErrNC$   & $\ErrOsc$  & $\Vert \uu - \uh \Vert$ & $\Err$  & $I_\mathrm{eff}$ \\ \hline

    4           & 12            & 2.598e{-5} & 1.246e{-5} & 2.600e{-5}              & 3.467e{-5} & 1.33
    \\ 16 & 48 & 4.066e{-7} & 4.897e{-8} & 4.066e{-7} & 4.356e{-7} & 1.07
    \\ 64 & 192 & 6.354e{-9} & 1.914e{-10} & 6.354e{-9} & 6.462e{-9} & 1.02
    \\ 256 & 768 & 9.928e{-11} & 7.475e{-13} & 9.928e{-11} & 9.97e{-11} & 1.00
    \\ 1024 & 3072 & 1.551e{-12} & 2.920e{-15} & 1.552e{-12} & 1.553e{-12} & 1.00

    \\ \hline \hline

    \multicolumn{7}{|c|}{\bm{$\poldeg = 3, \poldegp= 3$} } \\ \hline
    \# Elements & \# DOF($\uh$) & $\ErrNC$   & $\ErrOsc$  & $\Vert \uu - \uh \Vert$ & $\Err$  & $I_\mathrm{eff}$ \\ \hline

    4           & 16            & 7.852e{-7} & 5.666e{-7} & 7.859e{-7}              & 1.271e{-6} & 1.62
    \\ 16 & 64 & 3.085e{-9} & 5.579e{-10} & 3.085e{-9} & 3.503e{-9} & 1.14
    \\ 64 & 256 & 1.205e{-11} & 5.451e{-13} & 1.205e{-11} & 1.245e{-11} & 1.03
    \\ 256 & 1024 & 4.718e{-14} & 5.324e{-16} & 4.730e{-14} & 4.756e{-14} & 1.01

    \\ \hline
    \multicolumn{7}{|c|}{\bm{$\poldeg = 4, \poldegp= 4$} } \\ \hline
    \# Elements & \# DOF($\uh$) & $\ErrNC$   & $\ErrOsc$  & $\Vert \uu - \uh \Vert$ & $\Err$  & $I_\mathrm{eff}$ \\ \hline

    4           & 20            & 2.847e{-8} & 2.666e{-8} & 2.851e{-8}              & 5.133e{-8} & 1.80
    \\ 8 & 40 & 8.957e{-10} & 4.204e{-10} & 8.959e{-10} & 1.217e{-9} & 1.36
    \\ 16 & 80 & 2.804e{-11} & 6.582e{-12} & 2.804e{-11} & 3.282e{-11} & 1.17
    \\ 32 & 160 & 8.765e{-13} & 1.029e{-13} & 8.765e{-13} & 9.482e{-13} & 1.08
    \\ 64 & 320 & 2.742e{-14} & 1.608e{-15} & 2.753e{-14} & 2.852e{-14} & 1.04
    \\ \hline
  \end{tabular}
  \label{Tab::conv_DPG}
\end{table}

\begin{table}[H]
  \centering
  \caption{Convergence of the error $\Vert \uu - \uh \Vert$, the error estimators $\Err$, $\ErrNC$, and $\ErrOsc$, and the effectivity indices $I_{\mathrm{eff}}$ for the {\DG} method~\eqref{Eq::DG} with different polynomial degrees $k$}
  \begin{tabular}{| c | c   | c |   c  |c |  c |  c |}
    \hline

    \multicolumn{7}{|c|}{\bm{$\poldeg = 1, \poldegp= 1$} } \\ \hline
    \# Elements & \# DOF($\uh$) & $\ErrNC$   & $\ErrOsc$  & $\Vert \uu - \uh \Vert$ & $\Err$  & $I_\mathrm{eff}$ \\ \hline

    4           & 8             & 3.048e{-3} & 3.074e-4   & 3.021e-3                & 3.327e-3   & 1.10
    \\ 16 & 32 & 1.906e-4 & 4.811e{-6} & 1.901e-4 & 1.949e-4 & 1.03
    \\ 64 & 128 & 1.191e{-5} & 7.518e{-8} & 1.190e{-5} & 1.198e{-5} & 1.01
    \\ 256 & 512 & 7.445e{-7} & 1.175e{-9} & 7.444e{-7} & 7.455e{-7} & 1.00
    \\ 1024 & 2048 & 4.653e{-8} & 1.836e{-11} & 4.653e{-8} & 4.655e{-8} & 1.00

    \\ \hline \hline

    \multicolumn{7}{|c|}{\bm{$\poldeg = 2, \poldegp= 2$} } \\ \hline
    \# Elements & \# DOF($\uh$) & $\ErrNC$   & $\ErrOsc$  & $\Vert \uu - \uh \Vert$ & $\Err$  & $I_\mathrm{eff}$ \\ \hline

    4           & 12          & 4.210e{-5} & 1.246e{-5} & 4.450e{-5}              & 4.843e{-5} & 1.19
    \\ 16 & 48 & 6.299e{-7} & 4.897e{-8} & 6.307e{-7} & 6.582e{-7} & 1.04
    \\ 64 & 192 & 9.844e{-9} & 1.914e{-10} & 9.847e{-9} & 9.951e{-9} & 1.01
    \\ 256 & 768 & 1.538e{-10} & 7.475e{-13} & 1.538e{-10} & 1.542e{-10} & 1.00
    \\ 1024 & 3072 & 2.403e{-12} & 2.921e{-15} & 2.403e{-12} & 2.405e{-12} & 1.00

    \\ \hline \hline

    \multicolumn{7}{|c|}{\bm{$\poldeg = 3, \poldegp= 3$} } \\ \hline
    \# Elements & \# DOF($\uh$) & $\ErrNC$   & $\ErrOsc$  & $\Vert \uu - \uh \Vert$ & $\Err$  & $I_\mathrm{eff}$ \\ \hline

    4           & 16            & 1.186e{-6} & 5.666e{-7} & 1.169e{-6}              & 1.658e{-6} & 1.42
    \\ 16 & 64 & 4.664e{-9} & 5.579e{-10} & 4.647e{-9} & 5.075e{-9} & 1.09
    \\ 64 & 256 & 1.822e{-11} & 5.451e{-13} & 1.821e{-11} & 1.861e{-11} & 1.02
    \\ 256 & 1024 & 7.172e{-14} & 5.324e{-16} & 7.181e{-14} & 7.210e{-14} & 1.00

    \\ \hline
    \multicolumn{7}{|c|}{\bm{$\poldeg = 4, \poldegp= 4$} } \\ \hline
    \# Elements & \# DOF($\uh$) & $\ErrNC$   & $\ErrOsc$  & $\Vert \uu - \uh \Vert$ & $\Err$  & $I_\mathrm{eff}$ \\ \hline

    4           & 20            & 4.240e{-8} & 2.666e{-8} & 4.252e{-8}              & 6.453e{-8} & 1.52
    \\ 8 & 40 & 1.335e{-9} & 4.204e{-10} & 1.336e{-9} & 1.645e{-9} & 1.23
    \\ 16 & 80 & 4.179e{-11} & 6.582e{-12} & 4.180e{-11} & 4.647e{-11} & 1.11
    \\ 32 & 160 & 1.307e{-12} & 1.029e{-13} & 1.307e{-12} & 1.377e{-12} & 1.05
    \\ 64 & 320 & 4.083e{-14} & 1.608e{-15} & 4.094e{-14} & 4.192e{-14} & 1.02
    \\ \hline
  \end{tabular}
  \label{Tab::conv_DG}
\end{table}

\section{Extension to multiple space dimensions}
\label{Sec::extend}
In this section, we investigate a possible extension of the ideas presented so far to the multi-dimensional case. We consider the advection equation~\eqref{Eq::adv_multiD} on a simply-connected Lipschitz polytope $\dom \subset \R^\dimx$ for $\dimx \geq 2$. The velocity field $\vel(x) \in \Cnt^1(\overline \dom; \R^\dimx)$ is considered to be divergence-free. We also assume that $\vel$ is  \emph{$\dom$-filling}, \ie, its trajectories starting from the inflow boundary $\partial_{-} \dom$ fill $\overline \dom$ almost everywhere in a finite time. A sufficient condition for the validity of this property is given by~\cite{azerad1996inegalite} (see Lemma~\ref{Lem::om_fil} below). One can find necessary and sufficient conditions in~\cite[Lemma 2.3]{devinatz1974asymptotic}, see also~\cite{ayuso2009discontinuous, dahmen2012adaptive, cantin2017well, cantin2017edge}.

%

\subsection{Spaces}

We start by introducing proper generalizations of~\eqref{eq_Hpm}.
Let us define the operator related to~\eqref{Eq::adv_multiD} and its formal adjoint as
\begin{equation*}
  \linop \colon {\vv} \mapsto \inprod{\vel}{\grad \vv  }, \qquad \qquad  \linop^* \colon {\vv} \mapsto -\divg \lf( \vel {\vv}  \rt) {= - \inprod{\vel}{\grad \vv}},
\end{equation*}
together with the following graph spaces
\begin{align*}
  \Hgrph{\dom} & \eq \lf\{ \vv \in \Ltwo(\dom), \, \linop \vv \in \Ltwo(\dom)  \rt\}, &   & \Hgrphad{\dom} \eq \lf\{ \vv \in \Ltwo(\dom), \, \linop^* \vv \in \Ltwo(\dom)  \rt\}.
\end{align*}
Then $\linop: \Hgrph{\dom} \to \Ltwo(\dom)$ and $\linop^*\colon \Hgrphad{\dom} \to \Ltwo(\dom)$, and $\Hgrph{\dom} = \Hgrphad{\dom}$.
Moreover, one can define the following subspaces of the graph spaces with incorporated boundary conditions:
\begin{align*}
  \Hgrphz{\dom}   & \eq  \lf\{ \vv \in \Hgrph{\dom},  \vv      = 0 \text{ on } {\bdomm} \rt\}, \\
  \Hgrphadz{\dom} & \eq  \lf\{ \vv \in \Hgrphad{\dom}, \vv  = 0 \text{ on } {\bdomp}  \rt\}.
\end{align*}
These definitions are consistent extensions from $\dimx=1$ in that the spaces $\Hgrph{\dom}$, $\Hgrphad{\dom}$, $\Hgrphz{\dom}$, $\Hgrphadz{\dom}$ become respectively $\Hs{1}(\dom)$, $\Hs{1}(\dom)$, $\Hbminus{\dom}$, and $\Hbplus{\dom}$. One might confer with \cite[p.~131]{suli1999posteriori} and \cite[Theorems~2.1 and 2.2]{houston1999posteriori} for the justification of the trace operator which is discussed as an operator from $\Hgrph{\dom}$ to $\Hs{-\frac{1}{2}}(\bdomm)$ (or from $\Hgrphad{\dom}$ to $\Hs{-\frac{1}{2}}(\bdomp)$, respectively). The extension to $\Ltwo(\vert \inprod{\vel}{\nn} \vert; \bdomm)$ is possible under slightly more restrictive conditions, see~\cite[page~133]{suli1999posteriori}, \cite[Lemma~3.1]{ern2006discontinuous} and more recently \cite[Proposition~2.3]{dahmen2012adaptive}. Moreover the following integration-by-parts formula holds true:
\begin{equation}
  \label{Eq::IBP_multiD}
  \integprod{\vv}{\inprod{\vel}{\grad \ww}}{} + \integprod{\inprod{\vel}{\grad \vv}}{\ww}{} = \integprod{\inprod{\vel}{\nn} \vv}{\ww}{} \qquad \forall \vv \in \Hgrph{\dom}, \quad     \forall \ww \in \Hs{1}(\dom).
\end{equation}
The result~\eqref{Eq::IBP_multiD} can be extended to $\ww \in \Hgrphad{\dom}$.

\subsection{Streamline Poincar\'e inequality}

The following sufficient condition for the field $\vel$ to be $\dom$-filling is given in~\cite{azerad1996inegalite}:

\begin{lem}[$\dom$-filling sufficient condition]
  \label{Lem::om_fil}
  Let $\vel \in \Cnt^1(\overline \dom; \R^\dimx)$ and assume that  there is a fixed unit vector $\bm k \in \R^\dimx$ and a real number $\alpha > 0$ such that
  \begin{equation}
    \label{Eq::filling}
    \forall x \in \overline \dom, \qquad \inprod{\vel}{\bm k} \geq \alpha.
  \end{equation}
  Then $\vel$ is $\dom$-filling.
\end{lem}
For $\dom$-filling $\vel$, we can extend the inequality~\eqref{Eq::Friedrichs} along the flow of $\vel$, \cf~\cite{azerad1996inegalite}:
\begin{lem}[streamline Poincar\'e inequality]
  \label{Lem::Friedrichs_stream}
  Let the field $\vel \in \Cnt^1(\overline \dom; \R^\dimx)$ be divergence-free and $\dom$-filling. Then there exists a streamline Poincar\'e constant $\cpcrv$ such that
  \begin{equation}
    \label{Eq::poincare}
    \Ltwonorm{\vv}{} \leq \cpcrv \Ltwonorm{\inprod{\vel}{\grad \vv}}{} \qquad {\forall} \vv \in \Hgrphz{\dom}.
  \end{equation}
  The constant $\cpcrv$ is bounded by $\cpcrv\leq 2 T$, where $T$ is the longest time that trajectories of the field $\vel$ spend in the domain $\dom$. In particular, $T \leq \mathrm{diam}(\dom)/{\alpha}$ under the assumption~\eqref{Eq::filling}.
\end{lem}
A similar result can also be obtained for a non divergence-free field, see~\cite{axelsson2014robust}. In the case where the field $\vel$ is constant, one can easily set $\bm k$ as the direction of the flow and $\alpha = \vert \vel \vert$.
A crucial consequence of Lemma~\ref{Lem::Friedrichs_stream} is that one can equip the spaces $\Hgrphz{\dom}$ and $\Hgrphadz{\dom}$ with the norm $\Vert \inprod{\vel}{\grad (\cdot)} \Vert$.

\begin{rem}[functions with mean value zero]
  \label{Rem::curved}
  While, following from Lemma~\ref{Lem::Friedrichs_stream}, the streamline Poincar\'e inequality holds true for functions with zero trace on the inflow of an arbitrary domain $\DD$, such a result is not valid for functions with mean value zero as a variant of the Poincar\'e inequality~\eqref{Eq::Poincare} in multiple spatial dimensions. This leads to significant differences in the analysis of the multi-dimensional case compared to one-dimensional one, and less rigorous results that we are able to present here.
\end{rem}

%

\subsection{Error-residual equivalence}
We consider the multi-dimensional extension of the ultra-weak formulation~\eqref{Eq::weak_two}:  find $\uu \in \Ltwo(\dom)$ such that
\begin{equation}
  \label{Eq::weak_MD}
  - \integprod{\uu}{\inprod{\vel}{\grad \vv}}{}  =  \integprod{\ff}{ \vv}{}, \qquad \forall \vv \in \Hgrphadz{\dom}.
\end{equation}
Define the residual operator $\Res(\uh) \in \Hgrphadz{\dom}'$ and its dual norm as in~\eqref{Eq::Res_two}--\eqref{Eq::Res_two_norm}, upon replacing $\Hbplus{\dom}$ by $\Hgrphadz{\dom}$.
One can extend the equivalence of Theorem~\ref{Thm::Err_Res} to the multi-dimensional case as follows:
\begin{thm}[error-residual equivalence]
  \label{Thm::Err_Res_MD} Let the field $\vel \in \Cnt^1(\overline \dom; \R^\dimx)$ be divergence-free and $\dom$-filling. Let $\uu \in \Ltwo(\dom)$ be the ultra-weak solution of~\eqref{Eq::weak_MD}.
  Then
  \[
    \Ltwonorm{\uu - \uh}{} = \Vert \Res(\uh) \Vert_{\vel; \, \Hgrphadz{\dom}'} \qquad \forall \uh \in \Ltwo(\dom).
  \]
\end{thm}
\begin{proof}
  %
  We use the fact that for all $\vv \in \Ltwo(\dom)$, there exists a unique $\zz \in \Hgrphadz{\dom}$ such that
  %
  \begin{equation*}
    -\integprod{\inprod{\vel}{\grad \zz}}{\ww}{}=\integprod{\vv}{\ww}{}, \qquad \forall \ww \in \Ltwo(\dom).
  \end{equation*}
  The rest of the proof goes along the lines of that of Theorem~\ref{Thm::Err_Res}.
\end{proof}

\subsection{Local problems and the error indicator}
\label{Sec::local_MD}
In this section, we propose a heuristic approach inspired by the rigorous discussions in the one-dimensional case.
First, let us consider the following reconstruction, mimicking Definition~\ref{Defi::local_second}. Here $\triag$ is a simplicial mesh of $\dom$, $\trinod$ the patch of all simplices which share the given vertex $\nod \in \nodh$, $\patch$ the corresponding open subdomain, and $\psia$ the associated hat basis function.
\begin{defi}[patchwise problems]
  \label{Defi::local_MD}
  Let $\uh \in \Ltwo(\dom)$. For all vertices $\nod \in \nodh$, let $\sha \in \Xha$ be the solution of the following least-squares problem on the patch subdomain $\patch$:
  %
  \begin{align}
    \label{Eq::Recons_MD}
    \sha \defeq \arg \min_{\vh \in \Xha} \Big \{ \Ltwonorm{ \psia(\uh - \vh)}{\patch}^2 + C_{\mathrm{opt}}^{{2}} \Ltwonorm{\ff\psia + \lf(\inprod{\vel}{\grad \psia}\rt) \uh -  \inprod{ \vel}{ \grad (\psia \vh)}}{\patch}^2 \Big \}.
  \end{align}
  %
  %
  For $\poldegp \geq 0$, we take the finite-dimensional space $\Xha \defeq  \polp{\poldegp}(\trinod) \cap \Hgrphz{\patch}$ when the vertex $\nod$ lies in the closure of the inflow boundary $\bdomm$, and $\Xha \defeq  \polp{\poldegp}(\trinod) \cap \Hgrph{\patch}$ otherwise. Here $C_{\mathrm{opt}} > 0$ is a constant to be chosen. The global reconstruction $\sh$ is defined by
  \begin{equation}
    \label{Eq::recons_glob_MD}
    \sh \defeq \sum_{\nod \in \nodh} \psia \sha,
  \end{equation}
  leading to $\sh \in \polp{\poldegp + 1}(\triag) \cap \Hgrphz{\dom}$.
\end{defi}

In order to see the rationale behind the above reconstruction, one might note the following upper bound on the error exploiting Theorem~\ref{Thm::Err_Res_MD}, the integration-by-parts formula~\eqref{Eq::IBP_multiD}, the Cauchy--Schwarz inequality, and the streamline Poincar\'e inequality~\eqref{Eq::poincare}: for any $\sh \in \Hgrphz{\dom}$, we have
\begin{align*}
  \Ltwonorm{\uu - \uh}{} & =  \Vert \Res(\uh) \Vert_{\vel; \, \Hgrphadz{\dom}'} = \sup_{\vv \in  \Hgrphadz{\dom} \setminus \{0\}}  \dfrac{(\ff - \inprod{ \vel}{ \grad \sh}, \vv)  + (\uh - \sh, \inprod{\vel}{\grad \vv})}{\Ltwonorm{\inprod{\vel}{\grad \vv}}{}} \\
  & \leq \Ltwonorm{\uh - \sh}{}  + \cpcrv  \Ltwonorm{\ff -  \inprod{ \vel}{ \grad \sh}}{}.
\end{align*}
Using that almost each point in $\dom$ belongs to $(d+1)$ patch subdomains $\patch$ and the partition of unity~\eqref{Eq::unity_partition}, the construction of $\sh$ via~\eqref{Eq::recons_glob_MD} gives the following upper bound:
\[
  \Ltwonorm{\uu - \uh}{} \leq \lf \{  2 (\dimx+1) \sum_{\nod \in \nodh} \Big[ \Ltwonorm{ \psia(\uh - \sha)}{\patch}^2 +  \cpcrv^2  \Ltwonorm{\ff\psia {+ \lf(\inprod{\vel}{\grad \psia}\rt) \uh} -  \inprod{ \vel}{ \grad (\psia \sha)}}{\patch}^2 \Big] \rt \}^{1/2}.
\]
In particular, the idea of adding $0=\sum_{\nod \in \nodh}\inprod{\vel}{\grad \psia} \uh$ is inspired by the analysis in the one-dimensional case. By comparison to Definition~\ref{Defi::local_MD} one can see that the least-squares problems~\eqref{Eq::Recons_MD} minimize contributions to the upper bound on the error, and a theoretically-motivated choice for $C_{\mathrm{opt}}$ would be $C_{\mathrm{opt}} = \cpcrv$. This in particular leads to the guaranteed estimate $\Ltwonorm{\uu - \uh}{} \leq \Err$ with
\begin{equation}
  \label{Eq::err_indic_MD}
  \Err \defeq \lf\{\sum_{\el \in \triag}  \Vert \uh - \sh \Vert_\el^2 \rt \}^{1/2} +  \cpcrv \lf\{\sum_{\el \in \triag}  \Vert \ff - \inprod{\vel}{\grad \sh} \Vert^2_\el \rt \}^{1/2}.
\end{equation}
Numerical experiments, however, show that the second term in~\eqref{Eq::err_indic_MD} does not converge with the right order.
This apparently comes from the special structure of the minimization term which cannot be approximated up to the projection error, in contrast to the one-dimensional case, where~\eqref{Eq::orthog} holds true. Congruently, the lack of the Poincar\'e inequality in the streamline form (see Remark~\ref{Rem::curved}) implies the loss of the scaling by the mesh element diameters $\hk$ in the second term in~\eqref{Eq::err_indic_MD}, compare with $\ErrOsc$ given by~\eqref{Eq::Err_Res} in one space dimension.

In order to rectify this issue, we heuristically replace $\cpcrv$ (which typically scales as $2 \mathrm{diam}(\dom) / \alpha$, see Lemma~\ref{Lem::Friedrichs_stream}) in the estimate $\Err$ of~\eqref{Eq::err_indic_MD} by local terms $\frac{C' \hk}{\alpha}$,  where one now needs to choose the constant $C'$.
Overall, we suggest the following modified error indicator
\begin{equation}
  \label{Eq::err_indic_MD_mod}
  \Err_\mathrm{mod} \defeq \lf\{\sum_{\el \in \triag} \lf( \ErrNCel{\el}^2 +  \ErrResk{\el}^2 \rt) \rt \}^{1/2}, \quad \ErrNCel{\el} =  \Vert \uh - \sh \Vert_\el, \quad \ErrResk{\el} = \dfrac{C' \hk}{ \alpha }  \Vert \ff - \inprod{\vel}{\grad \sh} \Vert_\el.
\end{equation}
In this setting, one has two free parameters to choose, $C_\mathrm{opt}$ for the local problems in~\eqref{Eq::Recons_MD} and $C'$ in~\eqref{Eq::err_indic_MD_mod}. We set below $C_\mathrm{opt} = 2 \mathrm{diam}(\dom) / \alpha$, as Lemmas~\ref{Lem::om_fil}--\ref{Lem::Friedrichs_stream} suggest, and $C' = 2$. Numerically, our results are actually not sensitive to the choice of the parameter $C_\mathrm{opt}$.

Let us stress that, unlike $\Err$ in~\eqref{Eq::err_indic_MD}, the modified indicator $\Err_\mathrm{mod}$ from~\eqref{Eq::err_indic_MD_mod} is not a guaranteed upper bound of the error; however, our numerical results presented below show its applicability.

\subsection{Numerical experiments}
\label{Sec::num_MD}
In this section, we provide some numerical tests in two space dimensions to illustrate the properties of the error indicator~\eqref{Eq::err_indic_MD_mod}, with $\sh$ constructed following Definition~\ref{Defi::local_MD}. The implementation is done in the framework of FreeFEM++~\cite{hecht2012new} and based on the scripts for the reconstruction-based a posteriori estimation by~\cite{tang}.

One might note that the reconstruction $\sh$ of Definition~\ref{Defi::local_MD}, lying in the space $\Hgrphz{\dom}$, possibly allows capturing the discontinuity that may appear in the exact solution $\uu$ across the streamlines. This is, however, only in reach if the triangulation is aligned with the streamlines. If this is not the case, the reconstruction $\sh$ actually lies in the smoother space $\Hs{1}(\dom)$. Below, we will consider two test cases, one with both the exact solution $\uu$ and the reconstruction $\sh$ lying in $\Hs{1}(\dom)$, and one with a discontinuous solution $\uu \in \Hgrphz{\dom}$, with the triangulation aligned with the streamlines, ideally including also the lines of discontinuity of the exact solution.

Let us consider $\dom = (0, 1)^2$ and uniformly refined structured triangulations aligned with the slope $45 \degree$.  We only test here the {\DG} method~\eqref{Eq::DG}, since it is the only method among those considered in Section~\ref{Sec::Example} which is stable and well-defined in multiple space dimensions.

\subsubsection{Smooth solution}
We apply the right-hand side $\ff$ such that the solution of~\eqref{Eq::adv_multiD} is
\begin{equation} \label{eq_sol_sm}
  \uu(x, y) = \sin(\pi x) \sin(\pi y),
\end{equation}
for different velocity fields $\vel$.
The results are presented in Tables~\ref{Tab::conv_DG_MD_p} and~\ref{Tab::conv_DG_MD_vel} for various polynomial degrees $\poldeg$, with the choice $\poldegp = \poldeg + 1$ in Definition~\ref{Defi::local_MD}. Here $\ErrNC \defeq \lf( \sumk \ErrNCel{\el}^2 \rt)^{1/2}$ and $\ErrReskall \defeq \lf( \sumk \ErrResk{\el}^2 \rt)^{1/2}.$ The error indicator performs well in actually providing the upper bound of the error and simultaneously not overestimating it excessively. Moreover, the efficiency results appear to be robust with respect to both the velocity field $\vel$ and the polynomial degree $\poldeg$.
Compared to Section~\ref{Sec::Exp_vel}, both $\uu$ and $\uh$, but actually also $\sh$ constructed following Definition~\ref{Defi::local_MD}, turn out to be insensitive to the scaling of $\vel$ by a constant, so that the estimators in~\eqref{Eq::err_indic_MD_mod} do not change either.
In Figure~\ref{Fig::Error_DG_MD}, the distribution of the errors $\Ltwonorm{\uu - \uh}{\el}$ and of the error estimators $\Err_{\mathrm{mod}, \el} \defeq  \lf( \ErrNCel{\el}^2 +  \ErrResk{\el}^2 \rt)^{1/2}$ is presented. These distributions show a very close behavior, which suggests that the presented indicators should be suitable for adaptive mesh/polynomial degree refinement.

\begin{table}[H]
  \centering
  \caption{Convergence of the error $\Vert \uu - \uh \Vert$, the error estimators $\Err_\mathrm{mod}$, $\ErrNC$, and $\ErrReskall$, and the effectivity indices $I_{\mathrm{eff}}$ for   the {\DG} method~\eqref{Eq::DG}; smooth solution~\eqref{eq_sol_sm}, $\vel = (1, 1)^{\mathrm{t}}$, and different polynomial degrees $\poldeg$}
  \begin{tabular}{| c | c | c   | c |   c  |c |  c |  }
    \hline

    \multicolumn{7}{|c|}{\bm{$\poldeg = 1, \poldegp= 2$} } \\ \hline
    \# Elements & $\#$ DOF$(\uh)$ & $\ErrNC$  & $\ErrRes$ & $\Vert \uu - \uh \Vert$ & $\Err_\mathrm{mod}$ & $I_\mathrm{eff}$ \\ \hline

    8           & 24              & 9.365e-02 & 2.083e-01 & 1.097e-01               & 2.175e-01           & 1.98             \\
    32          & 96              & 2.584e-02 & 4.156e-02 & 2.963e-02               & 4.871e-02           & 1.64             \\
    128         & 384             & 6.786e-03 & 8.666e-03 & 7.553e-03               & 1.100e-02           & 1.46             \\
    512         & 1536            & 1.727e-03 & 1.983e-03 & 1.897e-03               & 2.630e-03           & 1.39             \\
    2048        & 6144            & 4.347e-04 & 4.773e-04 & 4.749e-04               & 6.456e-04           & 1.35             \\
    8192        & 24576           & 1.088e-04 & 1.173e-04 & 1.187e-04               & 1.601e-04           & 1.34             \\

    \hline \hline
    \multicolumn{7}{|c|}{\bm{$\poldeg = 2, \poldegp= 3$} } \\ \hline
    \# Elements & $\#$ DOF$(\uh)$ & $\ErrNC$  & $\ErrRes$ & $\Vert \uu - \uh \Vert$ & $\Err_\mathrm{mod}$ & $I_\mathrm{eff}$ \\ \hline

    8           & 48              & 2.271e-02 & 4.807e-02 & 1.882e-02               & 3.360e-02           & 1.78             \\
    32          & 192             & 3.106e-03 & 3.785e-03 & 2.476e-03               & 3.495e-03           & 1.41             \\
    128         & 768             & 3.972e-04 & 4.147e-04 & 3.135e-04               & 4.254e-04           & 1.36             \\
    512         & 3072            & 4.995e-05 & 5.012e-05 & 3.929e-05               & 5.280e-05           & 1.34             \\
    2048        & 12288           & 6.253e-06 & 6.216e-06 & 4.934e-06               & 6.592e-06           & 1.33             \\
    8192        & 49152           & 7.822e-07 & 7.843e-07 & 6.270e-07               & 8.322e-07           & 1.32             \\

    %

    \hline
  \end{tabular}
  \label{Tab::conv_DG_MD_p}
\end{table}

\begin{table}[H]
  \centering
  \caption{Convergence of the error $\Vert \uu - \uh \Vert$, the error estimators $\Err_\mathrm{mod}$, $\ErrNC$, and $\ErrReskall$, and the effectivity indices $I_{\mathrm{eff}}$ for the {\DG} method~\eqref{Eq::DG}; smooth solution~\eqref{eq_sol_sm}, different velocity fields, and $\poldeg =1$}
  \begin{tabular}{| c | c | c   | c |   c  |c |  c |  }
    \hline
    \multicolumn{7}{|c|}{\bm{$\poldeg = 1, \poldegp= 2, \vel=(100, 100)^{\mathrm{t}}$} } \\ \hline
    \# Elements & $\#$ DOF$(\uh)$ & $\ErrNC$  & $\ErrRes$ & $\Vert \uu - \uh \Vert$ & $\Err_\mathrm{mod}$ & $I_\mathrm{eff}$ \\ \hline

    8           & 24              & 9.365e-02 & 2.083e-01 & 1.097e-01               & 2.175e-01           & 1.98             \\
    32          & 96              & 2.584e-02 & 4.156e-02 & 2.963e-02               & 4.871e-02           & 1.64             \\
    128         & 384             & 6.786e-03 & 8.666e-03 & 7.553e-03               & 1.100e-02           & 1.46             \\
    512         & 1536            & 1.727e-03 & 1.983e-03 & 1.897e-03               & 2.630e-03           & 1.39             \\
    2048        & 6144            & 4.347e-04 & 4.773e-04 & 4.749e-04               & 6.456e-04           & 1.35             \\
    8192        & 24576           & 1.088e-04 & 1.173e-04 & 1.187e-04               & 1.601e-04           & 1.34             \\

    \hline \hline
    \multicolumn{7}{|c|}{\bm{$\poldeg = 1, \poldegp= 2, \vel=(10, 1)^{\mathrm{t}}$} } \\ \hline
    \# Elements & $\#$ DOF$(\uh)$ & $\ErrNC$  & $\ErrRes$ & $\Vert \uu - \uh \Vert$ & $\Err_\mathrm{mod}$ & $I_\mathrm{eff}$ \\ \hline

    8           & 24              & 8.299e-02 & 2.216e-01 & 1.009e-01               & 2.307e-01           & 2.28             \\
    32          & 96              & 2.057e-02 & 4.714e-02 & 2.896e-02               & 5.137e-02           & 1.77             \\
    128         & 384             & 5.325e-03 & 1.062e-02 & 7.965e-03               & 1.188e-02           & 1.49             \\
    512         & 1536            & 1.370e-03 & 2.684e-03 & 2.069e-03               & 3.014e-03           & 1.46             \\
    2048        & 6144            & 3.459e-04 & 6.807e-04 & 5.241e-04               & 7.636e-04           & 1.45             \\
    8192        & 24576           & 8.667e-05 & 1.711e-04 & 1.316e-04               & 1.918e-04           & 1.45             \\

    \hline \hline
    \multicolumn{7}{|c|}{\bm{$\poldeg = 1, \poldegp= 2, \vel=(y, x+1)^{\mathrm{t}}$} $(\alpha = 1)$ } \\ \hline
    \# Elements & $\#$ DOF$(\uh)$ & $\ErrNC$  & $\ErrRes$ & $\Vert \uu - \uh \Vert$ & $\Err_\mathrm{mod}$ & $I_\mathrm{eff}$ \\ \hline

    8           & 24              & 9.582e-02 & 2.239e-01 & 1.134e-01               & 2.360e-01           & 2.08             \\
    32          & 96              & 2.513e-02 & 5.212e-02 & 3.152e-02               & 5.780e-02           & 1.83             \\
    128         & 384             & 6.478e-03 & 1.233e-02 & 8.007e-03               & 1.393e-02           & 1.74             \\
    512         & 1536            & 1.636e-03 & 2.991e-03 & 2.013e-03               & 3.409e-03           & 1.69             \\
    2048        & 6144            & 4.103e-04 & 7.379e-04 & 5.053e-04               & 8.443e-04           & 1.67             \\
    8192        & 24576           & 1.027e-04 & 1.833e-04 & 1.267e-04               & 2.101e-04           & 1.66             \\

    \hline
  \end{tabular}
  \label{Tab::conv_DG_MD_vel}
\end{table}

\begin{figure}[H]
  \centering

  \begin{subfigure}[b]{\textwidth}
    \includegraphics[width=.49\textwidth]{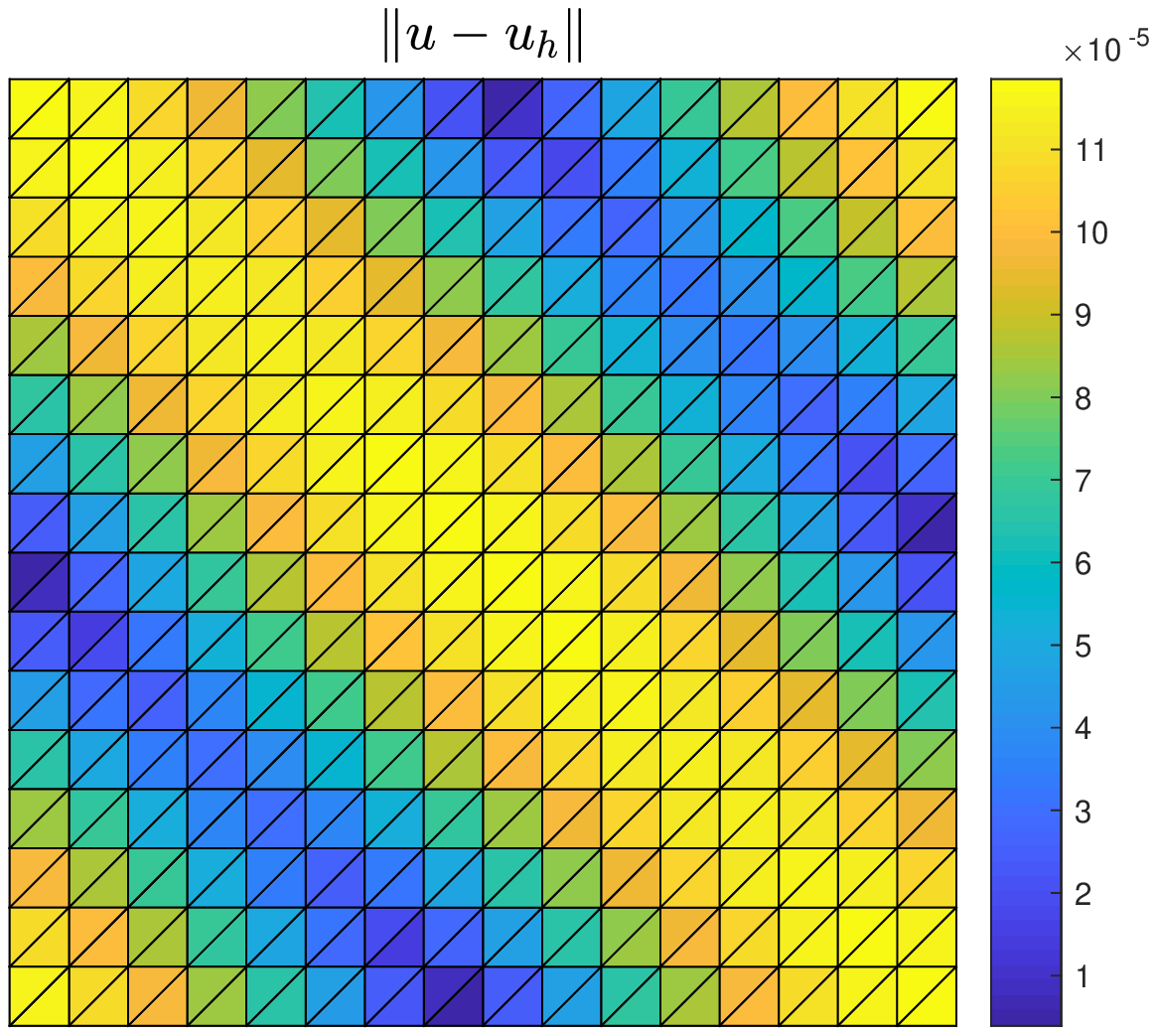}
    \includegraphics[width=.49\textwidth]{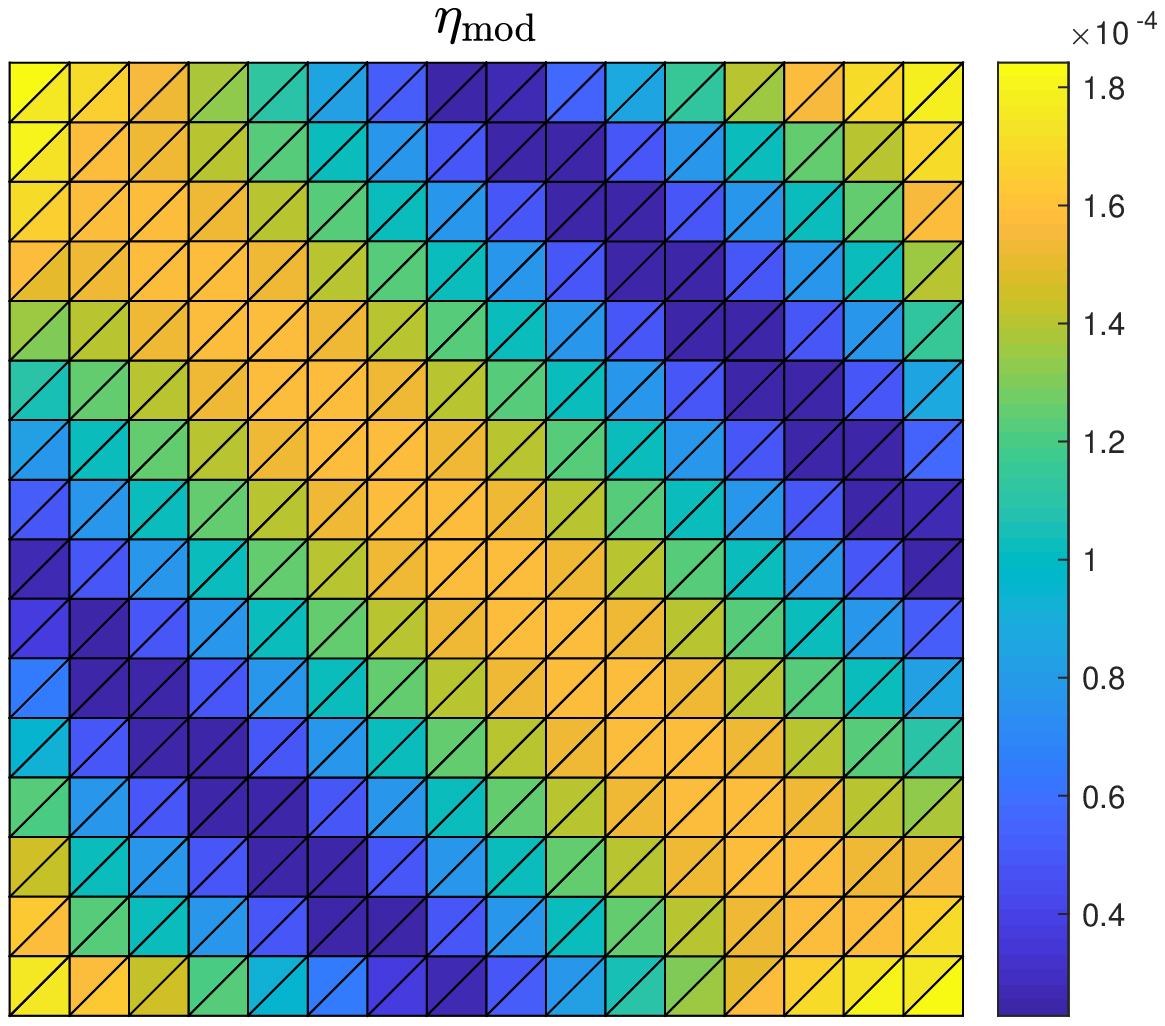}
    \caption{$\poldeg = 1,  \poldegp= 2$}
    \label{fig:trerr1}
  \end{subfigure}
  \\
  \begin{subfigure}[b]{\textwidth}
    \includegraphics[width=.49\textwidth]{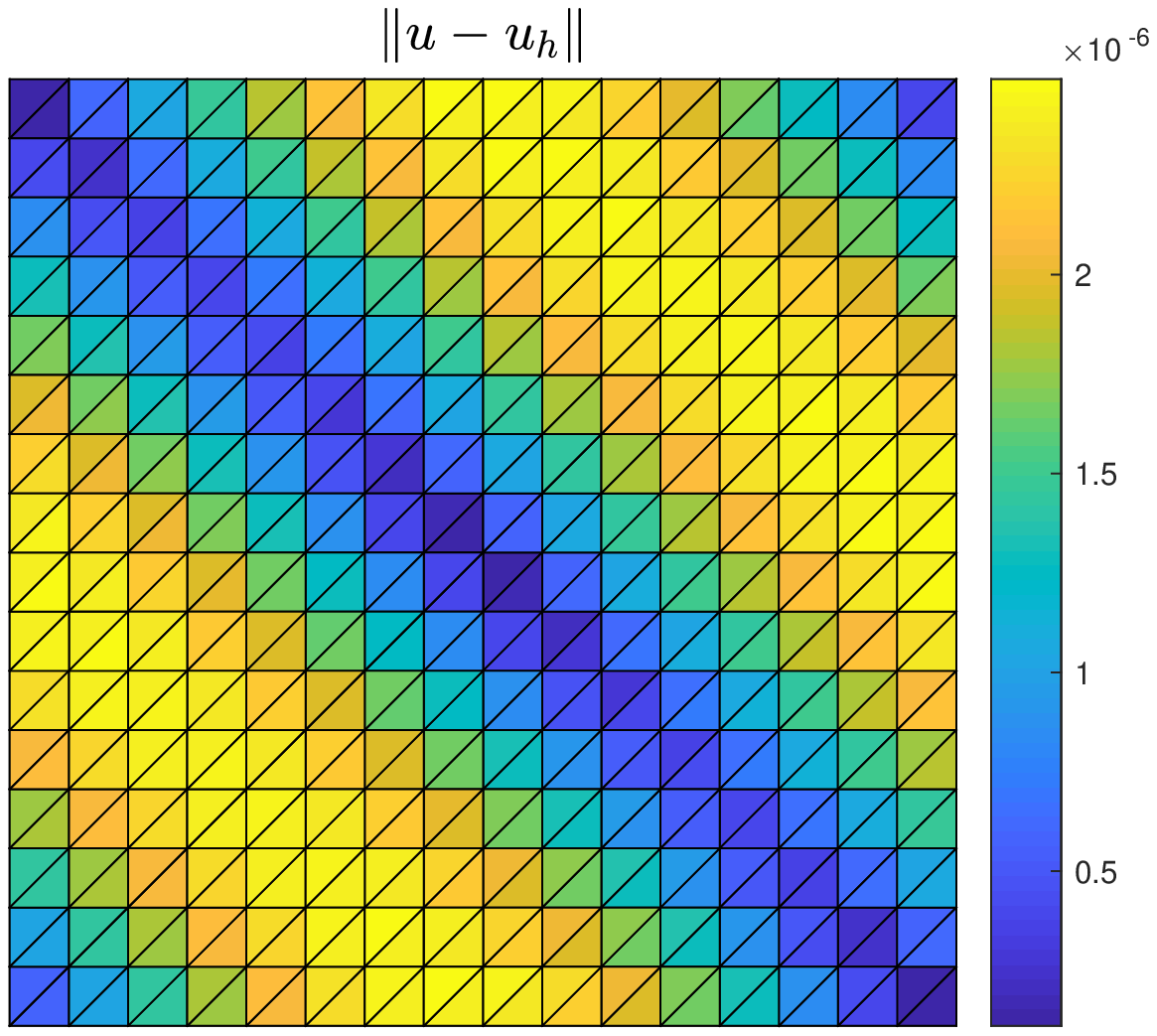}
    \includegraphics[width=.49\textwidth]{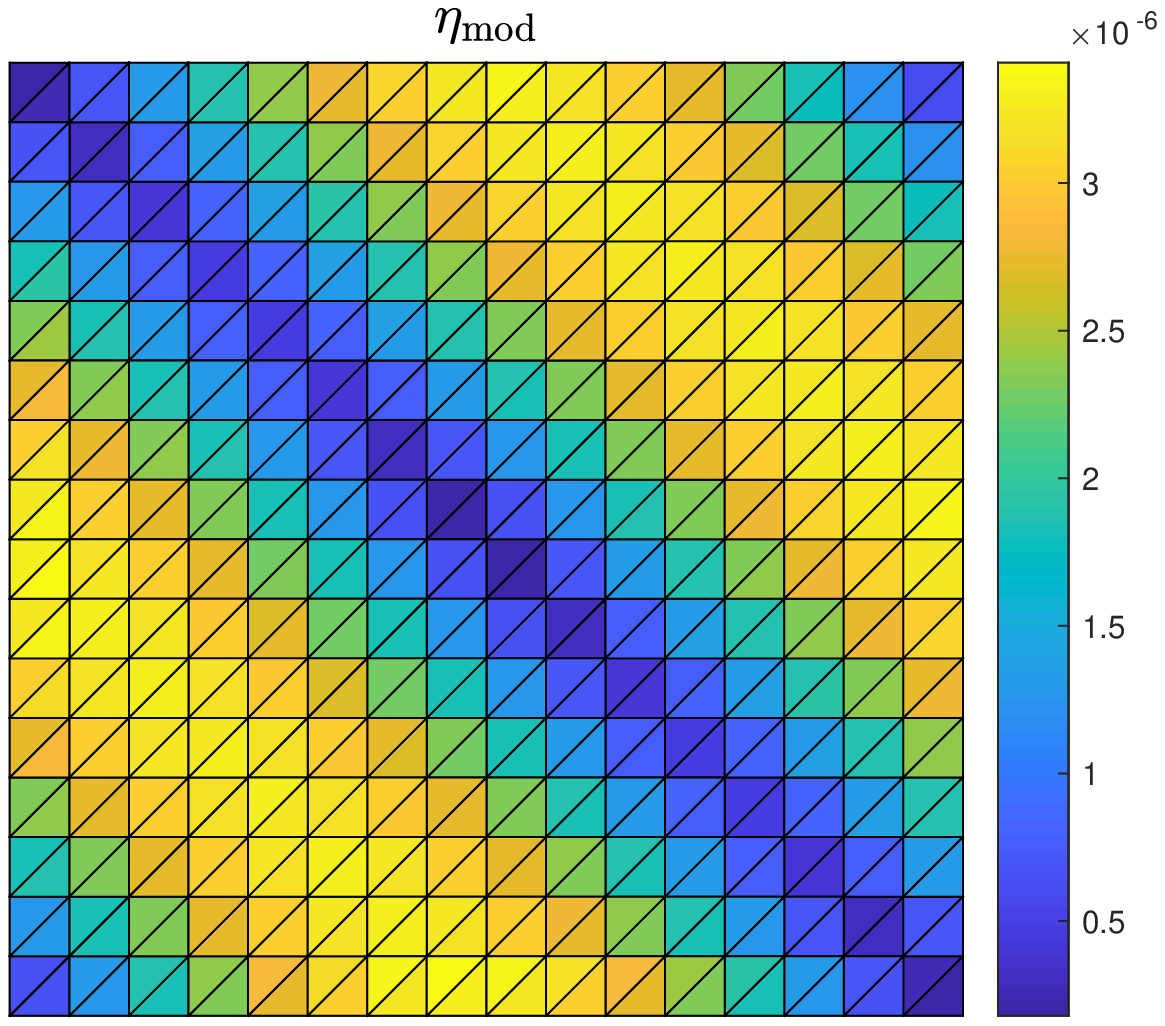}
    \caption{$\poldeg =2, \poldegp= 3$}
    \label{fig:trerr2}
  \end{subfigure}
  \caption{{Distribution of the errors $\Ltwonorm{\uu - \uh}{\el}$ (left) and of the local error estimators $\Err_{\mathrm{mod},\el}$ (right) for the {\DG} method~\eqref{Eq::DG} with $512$ elements; smooth solution~\eqref{eq_sol_sm}, $\vel=(1, 1)^{\mathrm{t}}$, and different polynomial degrees $\poldeg$}}

  \label{Fig::Error_DG_MD}
\end{figure}

\subsubsection{Discontinuous solution}

In this last example, we consider a discontinuous exact solution. For the velocity field $\vel = (1, 1)^{\mathrm{t}}$, we set
\begin{equation} \label{eq_sol_disc}
  \uu(x, y) = \begin{cases}
    0,                       & x < y, \\
    \sin(\pi x) \sin(\pi y), & x > y,
  \end{cases}
\end{equation}
and prescribe accordingly the right-hand side $\ff$. The triangulation is set to be aligned with this discontinuity; hence, the reconstruction $\sh$ is continuous everywhere but not at the discontinuity line of the exact solution. The results are presented in Table~\ref{Tab::conv_DG_MD_disc} for different polynomial degrees $\poldeg$. They show the robustness with respect to the polynomial degree of approximation. In Figure~\ref{Fig::Error_DG_MD_disc}, the distributions of the error and of the error estimators $\Err_{\mathrm{mod}, \el}$ are presented, again showing a very similar behavior.

\begin{table}[H]
  \centering
  \caption{Convergence of the error $\Vert \uu - \uh \Vert$, the error estimators $\Err_\mathrm{mod}$, $\ErrNC$, and $\ErrReskall$, and the effectivity indices $I_{\mathrm{eff}}$ for the {\DG} method~\eqref{Eq::DG}; discontinuous solution~\eqref{eq_sol_disc}, $\vel = (1, 1)^{\mathrm{t}}$, and different polynomial degrees $\poldeg$}
  \begin{tabular}{| c | c | c   | c |   c  |c |  c |  }
    \hline

    \multicolumn{7}{|c|}{\bm{$\poldeg = 1, \poldegp= 2$} } \\ \hline
    \# Elements & $\#$ DOF$(\uh)$ & $\ErrNC$  & $\ErrRes$ & $\Vert \uu - \uh \Vert$ & $\Err_\mathrm{mod}$ & $I_\mathrm{eff}$ \\ \hline

    8           & 24              & 6.622e-02 & 1.473e-01 & 7.758e-02               & 1.538e-01           & 1.98             \\
    32          & 96              & 1.827e-02 & 2.939e-02 & 2.095e-02               & 3.445e-02           & 1.64             \\
    128         & 384             & 4.798e-03 & 6.127e-03 & 5.341e-03               & 7.780e-03           & 1.45             \\
    512         & 1536            & 1.221e-03 & 1.402e-03 & 1.341e-03               & 1.860e-03           & 1.38             \\
    2048        & 6144            & 3.074e-04 & 3.375e-04 & 3.358e-04               & 4.565e-04           & 1.36             \\
    8192        & 24576           & 7.705e-05 & 8.295e-05 & 8.399e-05               & 1.132e-04           & 1.35             \\

    \hline \hline
    \multicolumn{7}{|c|}{\bm{$\poldeg = 2, \poldegp= 3$} } \\ \hline
    \# Elements & $\#$ DOF$(\uh)$ & $\ErrNC$  & $\ErrRes$ & $\Vert \uu - \uh \Vert$ & $\Err_\mathrm{mod}$ & $I_\mathrm{eff}$ \\ \hline

    8           & 48              & 1.626e-02 & 7.463e-02 & 1.330e-02               & 3.680e-02           & 2.76             \\
    32          & 192             & 2.227e-03 & 5.102e-03 & 1.751e-03               & 3.105e-03           & 1.77             \\
    128         & 768             & 2.850e-04 & 3.544e-04 & 2.217e-04               & 3.454e-04           & 1.56             \\
    512         & 3072            & 3.583e-05 & 2.847e-05 & 2.778e-05               & 4.181e-05           & 1.50             \\
    2048        & 12288           & 4.485e-06 & 2.849e-06 & 3.489e-06               & 5.184e-06           & 1.48             \\
    8192        & 49152           & 5.610e-07 & 3.410e-07 & 4.433e-07               & 6.525e-07           & 1.47             \\

    %
    %
    %
    %

    \hline
  \end{tabular}
  \label{Tab::conv_DG_MD_disc}
\end{table}

\begin{figure}[]
  \centering
  \begin{subfigure}[b]{\textwidth}
    \includegraphics[width=.49\textwidth]{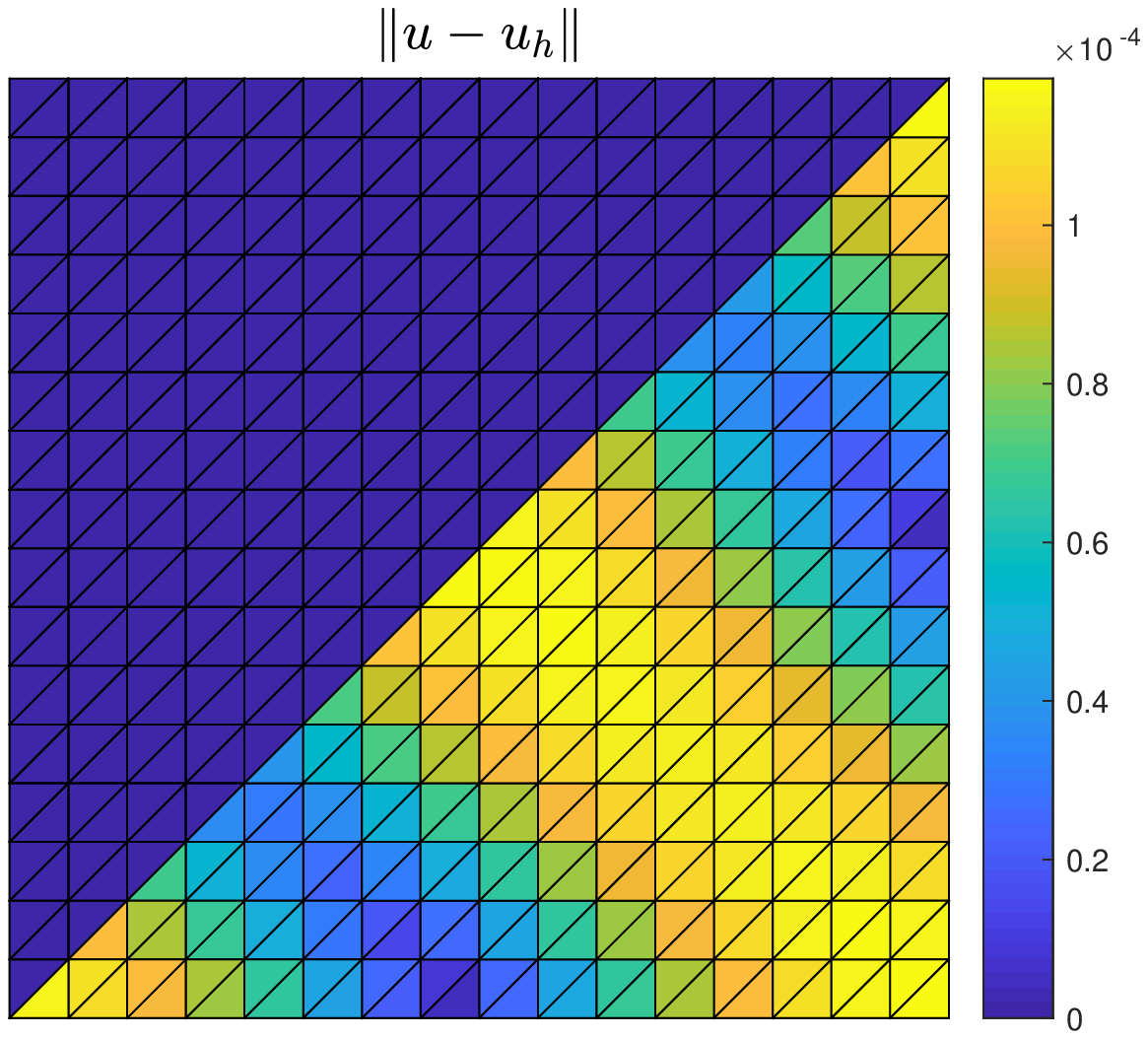}
    \includegraphics[width=.49\textwidth]{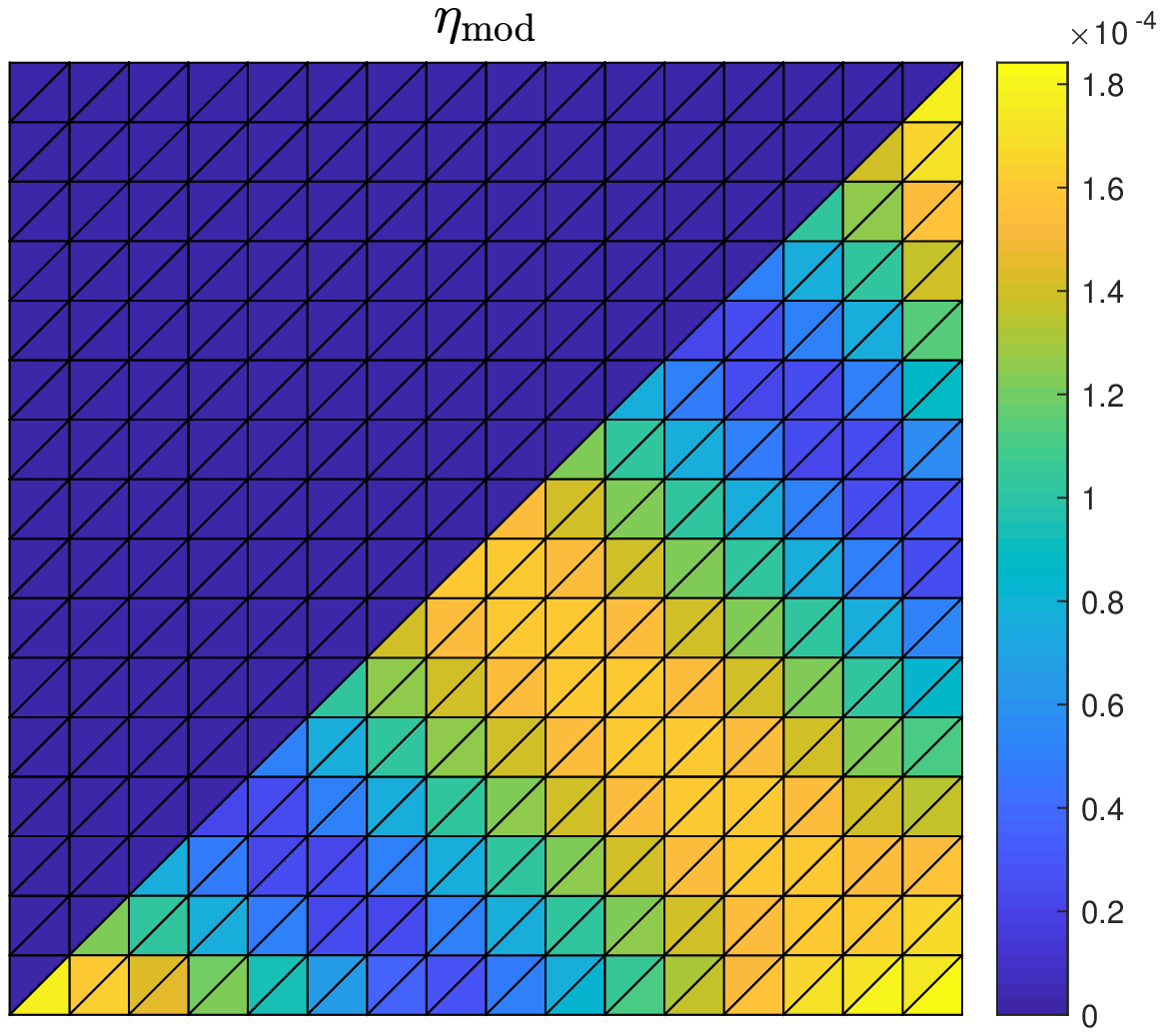}
    \caption{$\poldeg = 1,  \poldegp= 2$}
    \label{fig:trerr1_disc}
  \end{subfigure}
  \\
  \begin{subfigure}[b]{\textwidth}
    \includegraphics[width=.49\textwidth]{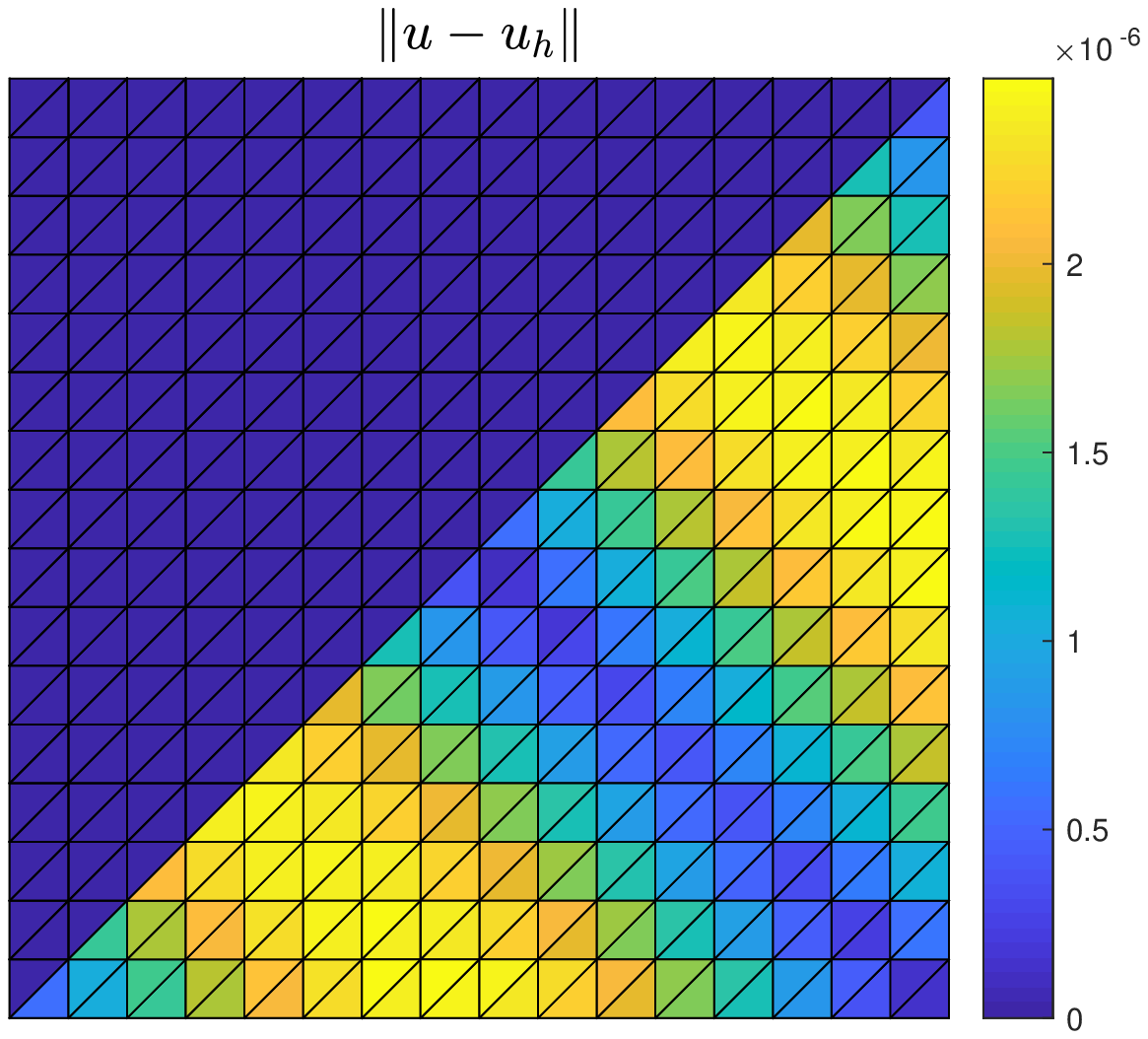}
    \includegraphics[width=.49\textwidth]{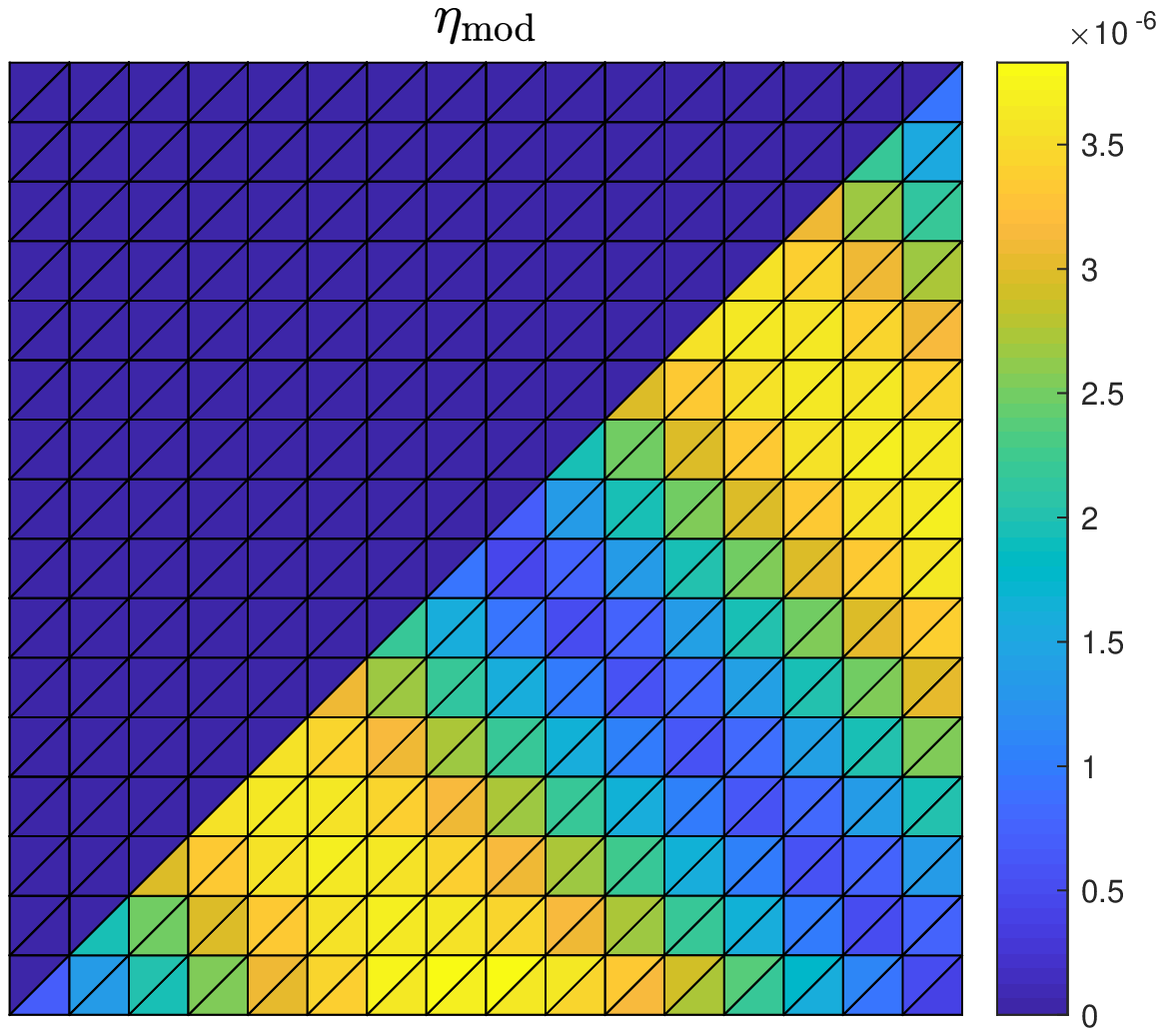}
    \caption{$\poldeg =2, \poldegp= 3$}
    \label{fig:trerr2_disc}
  \end{subfigure}
  \caption{{Distribution of the errors $\Ltwonorm{\uu - \uh}{\el}$ (left) and of the local error estimators $\Err_{\mathrm{mod},\el}$ (right) for the {\DG} method~\eqref{Eq::DG} with $512$ elements; discontinuous solution~\eqref{eq_sol_disc}, $\vel=(1, 1)^{\mathrm{t}}$, and different polynomial degrees $\poldeg$}}

  \label{Fig::Error_DG_MD_disc}
\end{figure}

\section{Conclusions}
\label{sec:conclusions}
In this work, we proposed a local reconstruction for numerical approximations of the one-dimensional linear advection equation, easily and independently obtained on each vertex patch. The reconstruction is proved to be well-posed and leads to a guaranteed upper bound of the $\Ltwo$-norm error between the actual solution $\uu$ and the approximation $\uh$. This error estimator is also proved to be locally efficient and robust with respect to both the advective field and the approximation polynomial degree.
These results hold in a unified framework that only requires the residual of $\uh$ to satisfy an orthogonality condition with respect to the hat basis functions. Numerical illustrations support the theory and additionally suggest asymptotic exactness.
Motivated by these results, a heuristic extension to any space dimension is presented, with numerical experiments in 2D being in line with those in 1D.

\bibliographystyle{siam}
\bibliography{biblio}
\end{document}

%% file: command.tex
\newtheorem{thm}{Theorem}[section]
\newtheorem{lem}[thm]{Lemma}
\newtheorem{defi}[thm]{Definition}
\newtheorem{asm}[thm]{Assumption}

\newtheorem{rem}[thm]{Remark}

\newtheorem{exam}[thm]{Example}
\numberwithin{equation}{section}

\ifesaim
\newtheorem{remark}{Remark}
\fi
\ifhal
\newtheorem{remark}{Remark}
\fi

\newif\ifmargin
\newif\ifdetail

\newcommand{\vel}{\bm b}		
\newcommand{\ff}{f}			
\newcommand{\dimx}{d}

\newcommand{\cpf}[1]{C_{\mathrm{PF}, #1}}	
\newcommand{\cp}[1]{C_{\mathrm{P}, #1}}		
\newcommand{\cF}[1]{C_{\mathrm{P}, #1}}		
\newcommand{\cpcrv}{C_{\mathrm{P}, \vel, \dom}}

\newcommand{\cregl}{C_{\mathrm{cont,  PF}}}

\newcommand{\uu}{u}

\newcommand{\vv}{v}

\newcommand{\ww}{w}

\newcommand{\zz}{z}

\newcommand{\uh}{\uu_\h}
\newcommand{\vh}{\vv_\h}

\newcommand{\sh}{s_\h}

\newcommand{\sha}{s_\h^\nod}

\newcommand{\Err}{\eta}				
\newcommand{\ErrOsc}{\Err_{\mathrm{Osc}}}	
\newcommand{\ErrOscf}[1]{\Err_{\mathrm{Osc},  {#1} }}	
	
\newcommand{\ErrResk}[1]{\Err_{\mathrm{R}, {#1} }}	
\newcommand{\ErrReskall}{\Err_{\mathrm{R}}}	
\newcommand{\ErrRes}{\Err_{\mathrm{R} }}	
\newcommand{\ErrNCel}[1]{\Err_{\mathrm{NC}, {#1} }}
\newcommand{\ErrNC}{\Err_{\mathrm{NC}}}

\newcommand{\sumk}{\sum_{\el \in \triag}}

\newcommand{\triag}{\mathcal{T}_\h}
\newcommand{\trinod}{\mathcal{T}_\nod}
\newcommand{\sklt}{\mathcal{E}_\h}
\newcommand{\edg}{e}

\newcommand{\h}{h}
\newcommand{\hk}{\h_\el}
\newcommand{\dom}{\Omega}
\newcommand{\DD}{D}

\newcommand{\bdom}{\partial \dom}
\newcommand{\bdomp}{\partial_+ \dom}
\newcommand{\bdomm}{\partial_- \dom}
\newcommand{\bdomz}{\partial_0 \dom}

\newcommand{\nn}{\bm n}
\newcommand{\el}{K}

\newcommand{\nod}{{\bm{ a}}}
\newcommand{\nodh}{\mathcal{V}_\h}
\newcommand{\nodhi}{\mathcal{V}_\h^{\mathrm{int}}}
\newcommand{\nodhin}{\mathcal{V}_\h^{\bdomm}}
\newcommand{\nodhout}{\mathcal{V}_\h^{\bdomp}}

\newcommand{\nodel}{\mathcal{V}_\el}
\newcommand{\patch}{{\omega_\nod}}
\newcommand{\psia}{\psi_\nod}
\newcommand{\card}{\mathcal{N}}

\newcommand{\R}{\mathbb{R}}
\newcommand{\Cnt}{\mathcal{C}}

\newcommand{\Lp}[1]{L_{#1}}
\newcommand{\Ltwo}{\Lp{2}}

\newcommand{\Hs}[1]{H^{#1}}

\newcommand{\Hdiv}[1]{\bm{H}(\mathrm{div}, #1)}
\newcommand{\Hgrph}[1]{H(\linop, #1)}
\newcommand{\Hgrphad}[1]{H(\linop^*, #1)}
\newcommand{\Hgrphz}[1]{H_0(\linop, #1)}
\newcommand{\Hgrphadz}[1]{H_0(\linop^*, #1)}

\newcommand{\Hbplus}[1]{H^1_{+}(#1)}
\newcommand{\Hbminus}[1]{H^1_{-}(#1)}
\newcommand{\Hbplusdual}[1]{H^{1}_{+}(#1){'}}

\newcommand{\polp}[1]{\mathcal{P}^{#1}}
\newcommand{\poldeg}{k}
\newcommand{\poldegp}{k'}

\newcommand{\X}{X}			
\newcommand{\Y}{Y}

\newcommand{\Ya}{{{H^1_\#}(\patch)}}	
\newcommand{\Yap}{{\vel;\, \Ya'}}	
\newcommand{\tYa}{{H^1_*}}	
\newcommand{\Xh}{\X_\h}	
\newcommand{\Yh}{\Y_\h}	
\newcommand{\Xha}{\X_\h^\nod}	
\newcommand{\Yha}{\Y_\h^\nod}	

\newcommand{\Res}{\mathcal{R}}
\newcommand{\B}{\mathcal{B}}
\newcommand{\inprod}[2]{{#1} {\cdot} {#2}}		
\newcommand{\linop}{\mathcal{L}}			
\newcommand{\grad}{\nabla}

\newcommand{\divg}{\nabla {\cdot}}

\newcommand{\integprod}[3]{( #1, #2 )_{#3}}
\newcommand{\integprodBigg}[3]{\Bigg( #1, #2 \Bigg)_{#3}}
\newcommand{\pair}[2]{\langle #1, #2 \rangle}

\newcommand{\Ltwonorm}[2]{\Vert #1 \Vert_{#2}}
\newcommand{\LtwonormBigg}[2]{\Bigg\Vert #1 \Bigg\Vert_{#2}}

\newcommand{\proj}{\Pi}
\newcommand{\jump}[1]{ \llbracket #1 \rrbracket }	
\newcommand{\avg}[1]{\{ \! \! \{ #1 \}\! \! \} }

\newcommand{\lf}{\left}
\newcommand{\rt}{\right}
\newcommand{\defeq}{\vcentcolon =}
\newcommand{\eqdef}{=\vcentcolon}

\newcommand{\DG}{{dG}}
\newcommand{\CPG}{{PG1}}
\newcommand{\DPG}{{PG2}}

\definecolor{brown(web)}{rgb}{0.65, 0.16, 0.16}

\ifsinum
\newcommand\ie{i.e.}

\newcommand\cf{cf.}
\fi

\ifhal
\newcommand\ie{i.e.}

\newcommand\cf{cf.}
\fi

\newcommand\eal{{\em et al. }}
\newcommand\eq{:=}

\newcommand\scp{{\cdot}} 

%% file: ex_shared_hal.tex
\title{Guaranteed and robust {$L_2$}-norm a posteriori error estimates for {1D} linear advection problems\thanks{This work was funded by the European Research Council (ERC) under the European Union’s Horizon 2020 research and innovation program (grant agreement No 647134 GATIPOR).}}

\renewcommand{\thefootnote}{\fnsymbol{footnote}}

\footnotetext[2]{Inria, 2 rue Simone Iff, 75589 Paris, France }
\footnotetext[3]{Universit\'e Paris-Est, CERMICS (ENPC), 77455 Marne-la-Vall\'ee,
France \\
\href{mailto:alexandre.ern@enpc.fr}{\texttt{alexandre.ern@enpc.fr}},
\href{mailto:martin.vohralik@inria.fr}{\texttt{martin.vohralik@inria.fr}}, 
\href{mailto:seyed-mohammad.zakerzadeh@inria.fr}{\texttt{seyed-mohammad.zakerzadeh@inria.fr}}
}

%

\author{ Alexandre Ern\footnotemark[3] \footnotemark[2] \and Martin Vohral\'ik\footnotemark[2] \footnotemark[3] \and  Mohammad Zakerzadeh\footnotemark[2] \footnotemark[3]}

\maketitle 
\begin{abstract}
\input{abstract}
\end{abstract}%

\bigskip

\noindent{\bf Key words:} linear advection problem; discontinuous Galerkin method; Petrov--Galerkin method; a posteriori error estimate; local efficiency;  advection robustness;  polynomial-degree robustness

%% file: abstract.tex
We propose a reconstruction-based a posteriori error estimate for linear advection problems in one space dimension. In our framework, a stable variational ultra-weak formulation is adopted, and the equivalence of the $\Ltwo$-norm of the error with the dual graph norm of the residual is established. This dual norm is showed to be localizable over vertex-based patch subdomains of the computational domain under the condition of the orthogonality of the residual to the piecewise affine hat functions. We show that this condition is valid for some well-known numerical methods including continuous/discontinuous Petrov--Galerkin and discontinuous Galerkin methods. Consequently, a well-posed local problem on each patch is identified, which leads to a global conforming reconstruction of the discrete solution. We prove that this reconstruction provides a guaranteed upper bound on the $\Ltwo$ error. Moreover, up to a constant, it also gives local lower bounds on the $\Ltwo$ error, where the generic constant is proven to be independent of mesh-refinement, polynomial degree of the approximation, and the advective velocity. This leads to robustness of our estimates with respect to the advection as well as the polynomial degree. All the above properties are verified in a series of numerical experiments, additionally leading to asymptotic exactness. Motivated by these results, we finally propose a heuristic extension of our methodology to any space dimension, achieved by solving local least-squares problems on vertex-based patches. Though not anymore guaranteed, the resulting error indicator is numerically robust with respect to both advection velocity and polynomial degree, for a collection of two-dimensional test cases including discontinuous solutions. 

%% file: ms.bbl
\begin{thebibliography}{10}

\bibitem{axelsson2014robust}
{\sc O.~Axelsson, J.~Kar\'atson, and B.~Kov\'acs}, {\em Robust preconditioning
  estimates for convection-dominated elliptic problems via a streamline
  {P}oincar\'e-{F}riedrichs inequality}, SIAM J. Numer. Anal., 52 (2014),
  pp.~2957--2976.

\bibitem{ayuso2009discontinuous}
{\sc B.~Ayuso and L.~D. Marini}, {\em Discontinuous {G}alerkin methods for
  advection-diffusion-reaction problems}, SIAM J. Numer. Anal., 47 (2009),
  pp.~1391--1420.

\bibitem{azerad1996inegalite}
{\sc P.~Az\'{e}rad and J.~Pousin}, {\em In\'{e}galit\'{e} de {P}oincar\'{e}
  courbe pour le traitement variationnel de l'\'{e}quation de transport}, C. R.
  Acad. Sci. Paris S\'{e}r. I Math., 322 (1996), pp.~721--727.

\bibitem{becker2013reconstruction}
{\sc R.~Becker, D.~Capatina, and R.~Luce}, {\em Reconstruction-based a
  posteriori error estimators for the transport equation}, in Numerical
  mathematics and advanced applications 2011, Springer, Heidelberg, 2013,
  pp.~13--21.

\bibitem{bey1996hp}
{\sc K.~S. Bey and J.~T. Oden}, {\em {$hp$}-version discontinuous {G}alerkin
  methods for hyperbolic conservation laws}, Comput. Methods Appl. Mech.
  Engrg., 133 (1996), pp.~259--286.

\bibitem{blechta2016localization}
{\sc J.~Blechta, J.~M{\'a}lek, and M.~Vohral{\'{\i}}k}, {\em Localization of
  the ${W}^{-1,q}$ norm for local a posteriori efficiency}, {IMA J. Numer.
  Anal.},  (2019).
\newblock DOI~10.1093/imanum/drz002.

\bibitem{Brae_Pill_Sch_p_rob_09}
{\sc D.~Braess, V.~Pillwein, and J.~Sch{\"o}berl}, {\em Equilibrated residual
  error estimates are {$p$}-robust}, Comput. Methods Appl. Mech. Engrg., 198
  (2009), pp.~1189--1197.

\bibitem{cantin2017well}
{\sc P.~Cantin}, {\em Well-posedness of the scalar and the vector
  advection-reaction problems in {B}anach graph spaces}, C. R. Math. Acad. Sci.
  Paris, 355 (2017), pp.~892--902.

\bibitem{cantin2017edge}
{\sc P.~Cantin and A.~Ern}, {\em An edge-based scheme on polyhedral meshes for
  vector advection-reaction equations}, ESAIM Math. Model. Numer. Anal., 51
  (2017), pp.~1561--1581.

\bibitem{Cars_Funk_full_rel_FEM_00}
{\sc C.~Carstensen and S.~A. Funken}, {\em Fully reliable localized error
  control in the {FEM}}, SIAM J. Sci. Comput., 21 (1999), pp.~1465--1484.

\bibitem{dahmen2012adaptive}
{\sc W.~Dahmen, C.~Huang, C.~Schwab, and G.~Welper}, {\em Adaptive
  {P}etrov-{G}alerkin methods for first order transport equations}, SIAM J.
  Numer. Anal., 50 (2012), pp.~2420--2445.

\bibitem{dahmen2019adaptive}
{\sc W.~Dahmen and R.~Stevenson}, {\em Adaptive strategies for transport
  equations}, arXiv preprint arXiv:1809.02055,  (2019).

\bibitem{devinatz1974asymptotic}
{\sc A.~Devinatz, R.~Ellis, and A.~Friedman}, {\em The asymptotic behavior of
  the first real eigenvalue of second order elliptic operators with a small
  parameter in the highest derivatives. {II}}, Indiana Univ. Math. J., 23
  (1973--1974), pp.~991--1011.

\bibitem{ern2006discontinuous}
{\sc A.~Ern and J.-L. Guermond}, {\em Discontinuous {G}alerkin methods for
  {F}riedrichs' systems. {I}. {G}eneral theory}, SIAM J. Numer. Anal., 44
  (2006), pp.~753--778.

\bibitem{ern2010guaranteed}
{\sc A.~Ern, A.~F. Stephansen, and M.~Vohral\'{i}k}, {\em Guaranteed and robust
  discontinuous {G}alerkin a posteriori error estimates for
  convection-diffusion-reaction problems}, J. Comput. Appl. Math., 234 (2010),
  pp.~114--130.

\bibitem{ern2015polynomial}
{\sc A.~Ern and M.~Vohral\'{i}k}, {\em Polynomial-degree-robust a posteriori
  estimates in a unified setting for conforming, nonconforming, discontinuous
  {G}alerkin, and mixed discretizations}, SIAM J. Numer. Anal., 53 (2015),
  pp.~1058--1081.

\bibitem{friedrichs1958symmetric}
{\sc K.~O. Friedrichs}, {\em Symmetric positive linear differential equations},
  Comm. Pure Appl. Math., 11 (1958), pp.~333--418.

\bibitem{georgoulis2014error}
{\sc E.~H. Georgoulis, E.~Hall, and C.~Makridakis}, {\em Error control for
  discontinuous {G}alerkin methods for first order hyperbolic problems}, in
  Recent developments in discontinuous {G}alerkin finite element methods for
  partial differential equations, vol.~157 of IMA Vol. Math. Appl., Springer,
  Cham, 2014, pp.~195--207.

\bibitem{georgoulis2019posteriori}
\leavevmode\vrule height 2pt depth -1.6pt width 23pt, {\em An a posteriori
  error bound for discontinuous {G}alerkin approximations of
  convection-diffusion problems}, IMA J. Numer. Anal., 39 (2019), pp.~34--60.

\bibitem{hecht2012new}
{\sc F.~Hecht}, {\em New development in {FreeFEM++}}, J. Numer. Math., 20
  (2012), pp.~251--265.

\bibitem{houston1999posteriori}
{\sc P.~Houston, J.~A. Mackenzie, E.~S\"uli, and G.~Warnecke}, {\em A
  posteriori error analysis for numerical approximations of {F}riedrichs
  systems}, Numer. Math., 82 (1999), pp.~433--470.

\bibitem{lax1960local}
{\sc P.~D. Lax and R.~S. Phillips}, {\em Local boundary conditions for
  dissipative symmetric linear differential operators}, Comm. Pure Appl. Math.,
  13 (1960), pp.~427--455.

\bibitem{makridakis2006posteriori}
{\sc C.~Makridakis and R.~H. Nochetto}, {\em A posteriori error analysis for
  higher order dissipative methods for evolution problems}, Numer. Math., 104
  (2006), pp.~489--514.

\bibitem{muga2018discrete}
{\sc I.~Muga, M.~J. Tyler, and K.~van~der Zee}, {\em The discrete-dual
  minimal-residual method ({DDMRes}) for weak advection-reaction problems in
  {Banach} spaces}, arXiv preprint arXiv:1808.04542,  (2018).

\bibitem{sangalli2004analysis}
{\sc G.~Sangalli}, {\em Analysis of the advection-diffusion operator using
  fractional order norms}, Numer. Math., 97 (2004), pp.~779--796.

\bibitem{sangalli2005uniform}
\leavevmode\vrule height 2pt depth -1.6pt width 23pt, {\em A uniform analysis
  of nonsymmetric and coercive linear operators}, SIAM J. Math. Anal., 36
  (2005), pp.~2033--2048.

\bibitem{sangalli2008robust}
\leavevmode\vrule height 2pt depth -1.6pt width 23pt, {\em Robust a-posteriori
  estimator for advection-diffusion-reaction problems}, Math. Comp., 77 (2008),
  pp.~41--70.

\bibitem{schotzau2009robust}
{\sc D.~Sch\"{o}tzau and L.~Zhu}, {\em A robust a-posteriori error estimator
  for discontinuous {G}alerkin methods for convection-diffusion equations},
  Appl. Numer. Math., 59 (2009), pp.~2236--2255.

\bibitem{suli1999posteriori}
{\sc E.~S{\"u}li}, {\em A posteriori error analysis and adaptivity for finite
  element approximations of hyperbolic problems}, in An introduction to recent
  developments in theory and numerics for conservation laws
  ({F}reiburg/{L}ittenweiler, 1997), vol.~5 of Lect. Notes Comput. Sci. Eng.,
  Springer, Berlin, 1999, pp.~123--194.

\bibitem{tang}
{\sc Z.~Tang}, {\em \url{https://who.rocq.inria.fr/Zuqi.Tang/freefem++.html}},
  2015.

\bibitem{tartakoff1972regularity}
{\sc D.~S. Tartakoff}, {\em Regularity of solutions to boundary value problems
  for first order systems}, Indiana Univ. Math. J., 21 (1972), pp.~1113--1129.

\bibitem{verfurth2005robust}
{\sc R.~Verf\"{u}rth}, {\em Robust a posteriori error estimates for stationary
  convection-diffusion equations}, SIAM J. Numer. Anal., 43 (2005),
  pp.~1766--1782.

\bibitem{verfurthbook}
{\sc R.~Verf\"{u}rth}, {\em A posteriori error estimation techniques for finite
  element methods}, Numerical Mathematics and Scientific Computation, Oxford
  University Press, Oxford, 2013.

\end{thebibliography}
